\documentclass[10pt]{article}

\usepackage[margin=1.2in]{geometry}
\usepackage[colorlinks,unicode]{hyperref}
\usepackage[all]{xy}
\usepackage[nottoc,numbib]{tocbibind}
\usepackage{amsmath,amsthm,amsfonts,longtable,verbatim,multicol,amssymb,wasysym,setspace,graphicx,titlesec,ulem,mathtools}
\usepackage[inline]{enumitem}

\newcommand{\bb}{\medbreak}
\newcommand{\nt}{\noindent}

\newcommand{\Q}{\mathbb{Q}}

\newcommand{\N}{\mathbb{N}}
\newcommand{\Cc}{\mathbb{C}}
\newcommand{\Kk}{k}
\newcommand{\rt}{\xrightarrow{}}
\newcommand{\xrt}{\xrightarrow}

\newcommand{\id}{\text{id}}

\newcommand{\GL}{\text{GL}}

\newcommand{\End}{\text{End}}
\newcommand{\tr}{\text{Tr}}

\newcommand{\ind}{\mathrm{Ind}}
\newcommand{\res}{\mathrm{Res}}

\newcommand{\rank}{\text{rank}}

\newcommand{\Alg}{\mathop{\mathrm{Alg}}}

\newcommand{\Mod}{\text{-}\mathrm{Mod}}

\newcommand{\tensor}{\otimes}
\newcommand{\F}{\mathcal{F}}
\newcommand{\HH}{\mathcal{H}}
\newcommand{\define}[1]{\textbf{#1}}

\newtheorem{lemma}{Lemma}[section]
\newtheorem{theorem}[lemma]{Theorem}
\newtheorem{cor}[lemma]{Corollary}
\newtheorem{conj}[lemma]{Conjecture}

\newtheorem{proposition}[lemma]{Proposition}

\theoremstyle{definition}
\newtheorem{definition}[lemma]{Definition}

\newtheorem{remark}[lemma]{Remark}

\titleformat{\section}
  {\centering\normalfont\fontsize{10}{15}\bfseries}{\thesection}{1em}{}
 
\titleformat{\subsection}
  {\normalfont\fontsize{11}{15}\bfseries}{\thesubsection}{1em}{}
  
\titleformat{\subsection}[runin]{\normalfont\bfseries}{\thesubsection.}{1em}{}{}
\titlespacing{\subsection}{0pt}{5pt}{10pt}

\title{Twists of representations of complex reflection groups and rational Cherednik algebras}
\author{Y. Bazlov and E. Jones-Healey}
\date{}

\begin{document}

\maketitle

\begin{abstract}
Drinfeld twists, and the twists of Giaquinto and Zhang, allow for algebras and their modules to be deformed by a cocycle. We prove general results about cocycle twists of algebra factorisations and induced representations and apply them to reflection groups and rational Cherednik algebras. In particular, we describe how a twist acts on characters of Coxeter groups of type $B_n$ and $D_n$ and relate them to characters of mystic reflection groups. This is used to characterise twists of standard modules of rational Cherednik algebras as standard modules for certain braided Cherednik algebras. We introduce the coinvariant algebra of a mystic reflection group and use a twist to show that an analogue of Chevalley's theorem holds for these noncommutative algebras. We also discuss several cases where the negative braided Cherednik algebras are, and are not, isomorphic to rational Cherednik algebras. 
\end{abstract}

\tableofcontents

\section{Introduction}
The main theme of the present paper is the use of methods from quantum algebra to achieve results in representation theory. 
Placing an algebra $A$ and its representations in a category of 
modules over a quasitriangular Hopf algebra, or more generally 
a Hopf algebra equipped with a $2$-cocycle, leads to a deformation --- known as cocycle twist --- of the associative product on $A$, as well as of the representations of $A$. This can be used to produce new
representation-theoretic constructions and to uncover properties of known algebras and representations by realising them as
twists.\bb

\nt A family of examples of cocycle twists arises from our previous paper \cite{twistsrcas}, joint with Berenstein and McGaw.
In \cite{twistsrcas} we demonstrate an isomorphism $\phi$ between braided Cherednik algebras, constructed earlier, and cocycle twists 
of rational Cherednik algebras $H_c(G(m,p,n))$ of imprimitive 
complex reflection groups $G(m,p,n)$ with even $m$.
These algebras have an action of the elementary abelian $2$-group $T=T(2,1,n)$ whose 
group algebra $\Cc T$ carries a non-standard, cohomologically non-trivial quasitriangular structure $\mathcal F$
which is the cocycle we use. 
A strong property of a rational Cherednik algebra $H_c(G)$ is its PBW-type factorisation,
$S\otimes \Cc G \otimes S'$ into three subalgebras: polynomial algebras $S$, $S'$ and the group algebra of $G$. A consequence of this factorisation is that there is a class of representations of $H_c(G(m,p,n))$ called standard modules, which are induced from an irreducible representation of $\Cc G$. This motivates us to study how a cocycle twist affects algebra factorisation and induced representations. We do this in Section~\ref{sec:general_twists} below.\bb

\nt A result which follows from the theory developed in Section~\ref{sec:general_twists} which is not readily obvious is that 
if $\frac mp$ is even, then the algebra $H_c(G(m,p,n))$ is isomorphic to its twist by $\mathcal F$.
(Recall that, in the nomenclature of imprimitive complex reflection groups, $\frac mp$ must be an integer.)
This follows from a construction due to Kulish and Mudrov \cite{kulish_2011} which gives an isomorphism $\eta \colon A_{\mathcal F} \to A$ 
of $\Kk$-algebras whenever the twist $A_{\mathcal F}$ is via an \textit{adjoint} action of a $\Kk$-Hopf algebra on $A$.
The action of $T$ on $H_c(G(m,p,n))$ for even $\frac mp$ is indeed adjoint because it factors via the embedding $T=T(2,1,n)$ as a subgroup in $G(m,p,n)$.
Combining this new map $\eta$ with the isomorphism $\phi$
from \cite{twistsrcas}, 
we conclude that the braided Cherednik algebra of the mystic reflection group $\mu(G(m,p,n))$ is isomorphic to  $H_c(G(m,p,n))$.
This is an example where quantum algebra methods are used to establish a result in representation theory.\bb

\nt The isomorphism $\eta$ thus obtained between the rational and braided Cherednik algebras (part of Theorem~\ref{embedding_thm} below) is not fully compatible with the PBW factorisation, although it restricts to an isomorphism $\eta\colon \Cc G(m,p,n)_{\mathcal F}\to \Cc G(m,p,n)$. Hence, combined with the map $\phi$ 
constructed in \cite{twistsrcas}, which restricts to 
an isomorphism $\Cc \mu(G(m,p,n)) \xrightarrow{\phi} \Cc G(m,p,n)_{\mathcal F}$, we obtain an isomorphism
$$
\eta \phi \colon \Cc \mu(G(m,p,n)) \to \Cc G(m,p,n)
$$
between the group algebras of the mystic and the complex reflection group.\bb

\nt Despite having isomorphic group algebras over $\Cc$, the groups $\mu(G(m,p,n))$ and $G(m,p,n)$ may not be isomorphic. However, if $\frac mp$ is even, these two groups are the same subgroup of the group of $n\times n$ monomial matrices, hence $\eta \phi$ gives us an automorphism of the group algebra 
of $G(m,p,n)$. In particular, $\eta \phi$ 
induces a permutation of the set of irreducible characters of $G(m,p,n)$ for even~$\frac mp$.
We study this permutation in Section~\ref{twisting_characters_sec} in the case $m=2$, $p=1$, that is, for the Coxeter group of type $B_n$. The main result is Proposition~\ref{permutation_of_characters_b_prop} which says that the quantum automorphism $\eta\phi$ of $\Cc B_n$ induces the permutation
$$
\chi_{(\lambda,\mu)} \circ \eta\phi = 
\chi_{(\lambda,\mu^*)}
$$
where the irreducible characters are labelled by bipartitions $(\lambda,\mu)$ of $n$, and $\mu^*$ stands for the partition dual to $\mu$. \bb

\nt After formally presenting the result of the isomorphism 
between $H_c(G(m,p,n))$ and the corresponding negative braided Cherednik algebra as part of Theorem~\ref{embedding_thm}, we introduce a noncommutative analogue of the finite-dimensional coinvariant algebra
$S_G$ of a complex reflection group $G=G(m,p,n)$.
We show that the noncommutative coinvariant algebra $\underline S_W$ of the mystic reflection group $W=\mu(G(m,p,n))$ carries 
a regular representation of $W$, in an analogy to the classical Chevalley theorem for $S_G$. This is another example of a purely representation-theoretic result 
which is proved using quantum techniques: 
an important ingredient is a proof that the above isomorphism $\phi$ 
intertwines the $W$-action on $\underline S_W$ with the Giaquinto-Zhang twist of the $G$-action on $S_G$.\bb

\nt There is further evidence to support the view that 
these new algebras $\underline S_W$ should be treated as a ``quantum'' analogue of $S_G$. Namely, the noncommutative 
multiplication on $\underline S_W$
is the cocycle twist of the product on $S_G$, see Corollary~\ref{cor:twist}.
Recall that the coinvariant algebra $S_G$ appears as a factor in the triangular decomposition
of the finite-dimensional restricted rational Cherednik algebra 
$\overline H_c(G) \cong S_G \otimes \Cc G \otimes S_G'$.
It turns out that the twist is fully compatible 
with the quotient map from $H_{0,c}(G)$
onto $\overline H_c(G) $, which leads to the 
twisted version $\overline{\underline{H}}_c(W)$
of $\overline H_c(G) $, constructed in Section~\ref{sec:last}.\bb

\nt It is the construction of $\overline{\underline{H}}_c(W)$
which provides us with an example where 
a Cherednik-type algebra is not isomorphic to its cocycle twist. The example given in Proposition~\ref{prop:not_isom}
is for algebras of rank $2$ over the field $\Q$; in fact, 
the two $64$-dimensional algebras are forms of each other, i.e., they become isomorphic if the scalars are extended to $\Cc$. Since the twist is defined over $\Q$, this raises a question about a possible interplay between twisting 
and arithmetic phenomena. This will be explored in further work.




\section{Twists of algebras and representations}
\label{sec:general_twists}

\subsection{\texorpdfstring{$H$}{H}-module algebras and smash products.}
In this section we recall the notion of cocycle twist. We work in the general setting of Hopf algebras, using the notation of Majid \cite{majid_1995}. Let $H$ be a Hopf algebra over a field $k$ with coproduct $\triangle:H\rt H\otimes H$, counit $\epsilon:H\rt k$ and antipode $S: H\to H$. In this section assume all structures are linear over $k$, with tensor products taken over $k$ too. An important example of a Hopf algebra is the group algebra $k T$ for a finite group $T$, where $\triangle(t)=t\otimes t$, $\epsilon(t)=1$ and $S(t)=t^{-1}$
for all $t\in T$, with these maps extending linearly to $kT$.\bb 

\nt Actions of algebras and Hopf algebras will generally be denoted by $\rhd$, 
adorned if necessary. 
If $A$ is both an algebra, with product $m:A\otimes A\rt A$, and an $H$-module
where $H$ acts by $\rhd$, then $A$ is additionally an $H$-module algebra if $m$ is an $H$-module homomorphism and $h\rhd 1_A=\epsilon(h)1_A\ \forall h\in H$. Equivalently, $H$-module algebras can be seen as algebra objects in the category of $H$-modules.\bb

\nt In what follows, the smash product of an $H$-module algebra $A$ with $H$ will be denoted by $A\#H$. Recall this is the algebra with underlying vector space $A\otimes H$ containing $A\cong A\otimes 1_H$ and $H\cong 1_A\otimes H$ as subalgebras, and with cross-commutation relation:
\begin{equation}\label{eq:smash_relation}
	ha=(h_{(1)}\rhd a)\#h_{(2)}\qquad\forall h\in H,a\in A.
\end{equation} 
We use Sweedler notation $\triangle(h)=h_{(1)}\otimes h_{(2)}$ for the coproduct, where summation is understood here.

\subsection{The cocycle twist of an algebra.}

The cocycle twist involves deforming the algebra structure of an $H$-module algebra via a $2$-cocycle on $H$:
\begin{definition}[\cite{majid_1995}, Example 2.3.1 and Theorem 2.3.4]
\label{def:2-cocycle}
A \define{$2$-cocycle} on a Hopf algebra $H$ is an invertible $\F\in H\otimes H$ such that $(\F\otimes 1)\cdot (\triangle \otimes \id)(\F)=(1\otimes \F)\cdot (\id \otimes \triangle)(\F)$ and $(\epsilon \otimes \id)(\F)=1=(\id\otimes \epsilon)(\F)$. The Hopf algebra $H^\F$ is defined as having the same algebra structure, and counit, as $H$, but with coproduct and antipode:
$$\triangle_\F h = \F (\triangle h)\F^{-1},\ S_\F (h)=U\cdot S(h)\cdot U^{-1},\ \forall h\in H$$
where $U:=\F_1\cdot S(\F_2)$ and $\F=\F_1\otimes \F_2$ (summation suppressed).
\end{definition}

\nt We can now recall the notion of \define{cocycle twist} (also called Drinfeld twist) via the following result. Below let $H\Mod$ denote the category of left $H$-modules. 
\begin{proposition}[\cite{twistsrcas}, Proposition 4.3]\label{prop:twist_is_functorial}
A $2$-cocycle $\F\in H\otimes H$ for a Hopf algebra $H$ gives rise to the functor 
$$(\ )_\F: \Alg(H\Mod) \to \Alg(H^\F\Mod)$$
which takes an object $(A,m)$ to $(A,m_\F = m(\F^{-1}\rhd - ))$, and an arrow $(A,m)\xrightarrow{\phi}(B,m')$ to $(A,m_\F)\xrightarrow{\phi_\F}(B,m'_\F)$,
where $\phi_\F=\phi$ as $H$-module morphisms.
\end{proposition}

\subsection{Cocycle twists of representations.}\label{cocycle_twist_of_reps_sec}

Giaquinto and Zhang \cite{giaquinto1998bialgebra} show that, under the restriction of working within the category of $H$-modules, it is possible to twist a representation of an algebra into a representation of the cocycle twist of that algebra. We outline this procedure next. 
\begin{definition}[\cite{giaquinto1998bialgebra}, Definition 1.6]
If $A$ is an $H$-module algebra, 
define $(H,A)\Mod$ to be 
%
%
%
the category whose objects $V$ are both $H$-modules and $A$-modules where the action $\rhd_A\colon A\otimes V \to V$ is a morphism of $H$-modules. That is, the compatibility condition 
\begin{equation}\label{eq:compatibility}
h\rhd (a\rhd_A v)=(h_{(1)}\rhd a)\rhd_A (h_{(2)}\rhd v),
\qquad\forall h\in H,a\in A,v\in V
\end{equation}
is satisfied. Morphisms in $(H,A)\Mod$ are
simultaneously $A$-module and $H$-module maps.  
\end{definition}

\nt The following notion of twist, from Giaquinto and Zhang \cite{giaquinto1998bialgebra}, takes objects from $(H,A)\Mod$ to $(H^\F,A_\F)\Mod$, where $A_\F$ is the cocycle twist of $A$, i.e.\ the image of $A$ under the functor given in Proposition \ref{prop:twist_is_functorial}. Suppose $V\in (H,A)\Mod$, and define $V_\F=V$ as a vector space. Note that $A\otimes V$ is naturally an $H\otimes H$-module via the actions of $H$ on $A$ and $V$ respectively, and therefore $\F^{-1}\rhd$ defines an endomorphism of $A\otimes V$. The map 
\begin{equation}\label{eq:def_twist}
\rhd_\F:A_\F\otimes V_\F\rt V_\F, \qquad \rhd_\F:=\rhd_A\circ (\F^{-1}\rhd )
\end{equation}
defines an $A_\F$-module structure on $V_\F$. This is easily checked to be a well-defined action on applying the conditions of a $2$-cocycle in Definition \ref{def:2-cocycle}, and using \eqref{eq:compatibility}. \bb

\nt We emphasise that \eqref{eq:def_twist} only applies to 
$(H,A)$-modules and not to arbitrary $A$-modules. 
The next Proposition implies that in general not all $A$-modules can be 
given an $(H,A)$-module structure.
A simple description of the $(H,A)$-module category is as follows:

\begin{proposition}[\cite{etingof2015tensor}, Exercise 7.8.32]
	\label{prop:category_equivalence_smash}
The category of $(H,A)$-modules is equivalent to $A\#H\Mod$.
\end{proposition}
\begin{proof}[Sketch of proof] 
On an $(H,A)$-module $V$, 
the formula  $(a\# h)\rhd v=a\rhd_A(h\rhd v)$
gives a well-defined $A\# H$-action: 
\eqref{eq:compatibility} is the same as compatibility of this action with the cross-commutator relation \eqref{eq:smash_relation} of the smash product.	
\end{proof}

\nt By a result of Kulish and Mudrov, the smash product 
is stable under the twist:

\begin{proposition}[\cite{kulish_2011}, Proposition 2.5]
	\label{prop:isom_semidirect_products}
Write $\F^{-1}\in H\otimes H$ as $\F^{-1}=f'\otimes f''$ (summation understood). The map 
\begin{equation}\label{eq:isom_semidirect_products}
	A_\F \# H^\F \xrightarrow{\sim} A \# H, \qquad a\# h \mapsto (f'\rhd a) \# f''h,
\end{equation}
is an isomorphism of algebras.	
\end{proposition}

\nt This result provides an alternative means of showing the map $\rhd_\F$ defined by Giaquinto and Zhang in \eqref{eq:def_twist} is a well-defined action of $A_\F$ on~$V$: $A_\F=A_\F\# 1 \subseteq A_\F\# H^\F$ embeds as a subalgebra in $A\# H$ via the morphism
\eqref{eq:isom_semidirect_products}, and the map $\rhd_\F$ can be seen to coincide with the restriction of the $A\# H$-action on $V$, arising from Proposition~\ref{prop:category_equivalence_smash}, onto the image of $A_\F$.\bb

\nt Proposition \ref{prop:isom_semidirect_products} also directly explains the following equivalence of categories proved by Giaquinto and Zhang; they are equivalent to module categories over isomorphic algebras $A_\F \# H^\F$ and $A \# H$:
\begin{cor}[\cite{giaquinto1998bialgebra}, Theorem 1.7]
The categories $(H,A)\Mod$ and $(H^\F, A_\F)\Mod$ are
equivalent.
\end{cor}

\subsection{Twists via adjoint actions.}
One can say much more about the twist and the semidirect product when $A$ has the following special $H$-module structure:
\begin{definition}[\cite{montgomery_1993}, Example 7.3.3]\label{adjoint_action}
An action $\rhd$ of $H$ on $A$ is called \define{adjoint} (or \define{strongly inner}) if, for some algebra homomorphism $u\colon H\to A$, $\rhd$ is given by: 
\begin{equation*}
h\rhd a=u(h_{(1)}) a u(Sh_{(2)}).
\end{equation*}
\end{definition}
\begin{proposition}\label{twist_isom_prop}
Assume that $A$ is an $H$-module algebra with an adjoint action of $H$ via the homomorphism $u\colon H \to A$, and retain the above notation. Then,
\begin{itemize}
	\item[1.] The smash product $A\# H$ and the tensor product $A\otimes H$ (where the two factors commute) are isomorphic as algebras, via the map 
	$a\# h \mapsto au(h_{(1)}) \otimes h_{(2)}$.
		\item[2.] The map 
	\begin{equation}\label{eq:isom_twist}
	\eta\colon A_\F\xrightarrow{\sim} A, \qquad a\mapsto (f'\rhd a)u(f''),
	\end{equation} is an isomorphism of algebras.
  \item[3.] Let $V$ be an $A$-module, and denote the action of $A$ as $\rhd_A$. Define an action of $H$ on $V$ as $\rhd_H=\rhd_A\circ (u\otimes \id_V)$. Then $V$ is an object in $(H,A)\Mod$, and the twisted action $\rhd_\F:A_\F\otimes V_\F\rightarrow V_\F$, $\rhd_\F:=\rhd_A\circ (\F^{-1}\rhd )$ is equal to the pullback of $\rhd_A$ along $\eta$, i.e. $\rhd_\F=\rhd_A \circ (\eta\otimes \id_{V_\F})$.
\end{itemize}
\end{proposition}
\begin{proof}
Part~1 follows from \cite[Example 7.3.3]{montgomery_1993}, or see \cite[Theorem 3.18]{berenstein_schmidt} for a proof by explicit calculation. 
Part 2 is \cite[Theorem 2.1]{kulish_2011}. Alternatively, the isomorphism $\eta$ in~\eqref{eq:isom_twist} is obtained as the composite map $A_\F=A_\F\# 1 \subseteq A_\F\# H^\F \xrightarrow{\sim} A\#H \to A$, given by \eqref{eq:isom_semidirect_products} followed by the map $a\# h \mapsto au(h_{1})\epsilon(h_{(2)}) = au(h)$ which, by part~1, is a homomorphism of algebras.
For Part 3, we first must check equation \eqref{eq:compatibility} holds. Firstly note $h\rhd_H (a\rhd_A v)=u(h)a\rhd_A v$, whilst for $\triangle(h)=h_{(1)}\otimes h_{(2)}$,
\begin{align*}
(h_{(1)}\rhd a)\rhd_A (h_{(2)}\rhd v) & =u(h_{(1)(1)})au(S(h_{(1)(2)}))\rhd_A (u(h_{(2)})\rhd_A v)\\
& = u(h_{(1)})au(S(h_{(2)(1)}h_{(2)(2)}))\rhd_A v\\
& = u(\epsilon(h_{(2)})h_{(1)})a\rhd_A v = u(h)a\rhd_A v
\end{align*}
where in the second, third and fourth equalities we use coassociativity, and the antipode and counit axioms for Hopf algebras respectively. Now note $a\rhd_\F v=(f'\rhd a)\rhd_A (f''\rhd_H v)=(f'\rhd a)\rhd_A (u(f'')\rhd_A v)=\eta(a)\rhd_A v$, as required.
\end{proof}
 
\begin{remark}\label{pullback_rmk} An immediate corollary to Proposition \ref{twist_isom_prop}(3) is that, if $\rho:A\rt \End(V)$ is the representation corresponding to $\rhd_A$, and $\rho_\F$ is the representation corresponding to $\rhd_\F$, arising as a result of the Giaquinto and Zhang twist, then we have that $\rho_\F$ is given by the pullback of $\rho$ along $\eta$, i.e. $\rho_\F=\rho\circ \eta$. Additionally, if $\chi$ and $\chi_\F$ are the characters of $\rho$ and $\rho_\F$ respectively, then $\chi_\F=\chi\circ \eta$.
\end{remark}

\subsection{Twists of algebra factorisations.}
Later we wish to discuss rational Cherednik algebras, which, owing to their PBW property, are examples of algebra factorisations. In this section, we prove a general result that twisting preserves the algebra factorisation structure.

\begin{definition}
If $C$ is an algebra, then an \define{algebra factorisation} $C=A\cdot B$ is a pair of subalgebras $A$, $B$ of an associative unital 
$k$-algebra $C$ such that the restriction of the multiplication map $m\colon {C\tensor C}\to C$ to $A\tensor B$ 
is an isomorphism 
\begin{equation}
\label{eq:basic isom}
m|_{A\otimes B}\colon A\tensor B \xrightarrow{\sim} C
\end{equation}
of $k$-vector spaces. To an algebra factorisation there is associated a linear map $\Psi_C\colon B\otimes A \to A\otimes B$ given by
\begin{equation}\label{eq:Psi}
	\Psi_C\colon B\otimes A \hookrightarrow C \otimes C \xrightarrow{m} C \xrightarrow{(m|_{A\otimes B})^{-1}} A \otimes B.
\end{equation}
\end{definition}
\nt The map $\Psi_C$ obeys the ``generalised braiding'' equations given by Majid in \cite[Proposition 21.4]{majid_2002}.\bb

\nt If $H$ is a Hopf algebra, we consider an algebra factorisation in $H\Mod$ to be an algebra factorisation $C=A\cdot B$ in which $C$ is an $H$-module algebra, and $A$, $B$ are $H$-module subalgebras of $C$. In this case, the isomorphism \eqref{eq:basic isom} and the map $\Psi_C$  given by \eqref{eq:Psi} are $H$-module homomorphisms.
Furthermore, given a $2$-cocycle $\F\in H\otimes H$, the Drinfeld twists $A_{\F},\ B_{\F}$ of $A,\ B$ are clearly $H^\F$-module subalgebras of the Drinfeld twist $C_{\F}$ of $C$. Additionally, the following holds,
\begin{proposition}\label{prop:twist_of_factorisation}
If $C=A\cdot B$ is an algebra factorisation in $H\Mod$, then $C_{\F}$ is an algebra factorisation $C_{\F}=A_{\F}\cdot B_{\F}$ in $H^{\F}\Mod$, and $\Psi_{C_{\F}} = (\F\rhd)\circ \Psi_C\circ (\F^{-1}\rhd)$.
\end{proposition}
\begin{proof}
Recall from Proposition~\ref{prop:twist_is_functorial} that the product $m_{\F}$ on $C_{\F}$ is given by $m_{\F}=m\circ (\F^{-1} \rhd )$,
hence its restriction onto 
$A_{\F}\otimes B_{\F}$ is the composition
$m_{\F}|_{A_{\F}\otimes B_{\F}}
= m|_{A\otimes B} \circ (\F^{-1} \rhd )|_{A\otimes B}$ where both maps on the right-hand side are bijective.
Therefore, $m_{\F}|_{A_{\F}\otimes B_{\F}}\colon A_{\F}\otimes B_{\F} \to C_{\F}$ is bijective, as required.
Substituting the formulas for $m_{\F}$
and $m_{\F}|_{A_{\F}\otimes B_{\F}}$ in~\eqref{eq:Psi}, we obtain $\Psi_{C_{\F}}$ as stated.
\end{proof}

\subsection{Induced representations.}
Induced representations of associative algebras were defined by Higman \cite{higman_1955}, extending the notion from group theory. Given a subalgebra $B$ of an associative unital algebra $C$
(so that $C$ is a $C$--$B$ bimodule via the regular action), the induction functor 
\begin{equation*}
	\ind_B^C\colon B\Mod \to C\Mod, \qquad 
	\ind_B^C(V) = C\otimes_B V,
\end{equation*}
is a left adjoint of the restriction functor from $C$ to $B$, see~\cite{rieffel_1975}. If $C=A\cdot B$ is an algebra factorisation and $V$ is a $B$-module, one has 
\begin{equation*}
\ind_B^C(V) \cong A\otimes V
\end{equation*}
as vector spaces, and moreover, as left $A$-modules. The action of $C$ on $\ind_B^C(V)$ is given by
\begin{equation}\label{eq:action_on_induced}
\rhd_C \colon C\otimes A \otimes V 
\xrightarrow{(m|_{A\otimes B})^{-1} \otimes \id_{A \otimes V}}
A\otimes B\otimes A \otimes V 
\xrightarrow{\id_A \otimes \Psi_C \otimes \id_V}
A \otimes A \otimes B \otimes V 
\xrightarrow{m\otimes \rhd_B}
A \otimes V
\end{equation}
where $\rhd_B$ is the action of $B$ on $V$. As pointed out above, when $A\cdot B$ is an algebra factorisation in $H\Mod$, then \eqref{eq:basic isom} and \eqref{eq:Psi} are $H$-module homomorphisms. Therefore, if $V\in (H,B)\Mod$, then, in view of \eqref{eq:action_on_induced}, we see that $\ind_B^C(V)\in (H,C)\Mod$.\bb

\nt We now show that for algebra factorisations, the induction functor commutes with the twist functor. More precisely, 
\begin{proposition}\label{twisting_induced_rep_prop}
If $C=A\cdot B$ is an algebra factorisation in $H\Mod$ and 
$V$ is a $(H,B)$-module, then 
\begin{equation}
     \ind_{B_{\F}}^{C_{\F}}(V_{\F})
    \cong
    (\ind_B^C(V))_{\F} 
\end{equation}
as $C_{\F}$-modules, where $B_\F$ and $C_\F$ are Drinfeld twists, and $V_\F$ and $(\ind_B^C(V))_\F$ are Giaquinto and Zhang twists, as in \eqref{eq:def_twist}. The two $C_{\F}$-module structures on the underlying vector space $A\otimes V$ are intertwined by the map $\F^{-1}\rhd\colon A\otimes V \to A \otimes V$. 
\end{proposition}
\begin{proof}
Above we explained that $\ind_B^C(V)\in (H,C)\Mod$, so firstly we find that the Giaquinto and Zhang twist $(\ind_B^C(V))_\F$ is a well-defined object.\bb

\nt Now by Proposition~\ref{prop:twist_of_factorisation}, the algebra
$C_{\F}$ is generated by  $A_{\F}$ 
and $B_{\F}$, therefore it is enough to show that 
the two actions of $A_{\F}$ on the space $A\otimes V$ are intertwined by $\F^{-1}\rhd$, and to do the same for $B_{\F}$.\bb

\nt In the following diagram, the top row represents the action of $A_{\F}$ on the induced module 
$\ind_{B_{\F}}^{C_{\F}}(V_{\F})$: 
the composite arrow is $m_{\F}\otimes \id_V$, 
which is indeed the action of $A_{\F}$ on the free 
left module $A_{\F}\otimes V$. The bottom row 
is the action of $A$ on $\ind_B^C(V)$, twisted by $\F$.
$$\xymatrix@C=7em{A\otimes A\otimes V \ar[r]^{(\F^{-1}\rhd) \otimes \id_V} \ar[d]_{\id_A\otimes (\F^{-1}\rhd)} & A\otimes A\otimes V \ar[r]^{m\otimes \id_V} \ar[d]_{(\triangle \otimes \id)(\F^{-1})\rhd} & A\otimes V \ar[d]_{\F^{-1}\rhd} \\
A\otimes A\otimes V \ar[r]_{(\id \otimes \triangle)(\F^{-1})\rhd } & A\otimes A \otimes V \ar[r]_{m\otimes \id_V} & A\otimes V}$$
 The leftmost square of the diagram commutes by the cocycle equation, and the rightmost square commutes because 
$m\colon A\otimes A \to A$ is an $H$-module morphism. 
Hence the diagram commutes, proving that $\F^{-1}\rhd$ intertwines the two $A_{\F}$-actions 
on $A\otimes V$.\bb

\nt The second diagram deals with the two actions of $B_{\F}$. The top row consists in applying 
$\Psi_{C_{\F}}$, which by Proposition~\ref{prop:twist_of_factorisation} is $(\F\rhd)\circ \Psi_C\circ (\F^{-1}\rhd)$ to the first two legs, followed by the action 
$\rhd_{B_{\F}} = \rhd_B\circ (\F^{-1}\rhd)$
on~$V$. By~\eqref{eq:action_on_induced}, this gives the action of $B_{\F}$ on $\ind_{B_{\F}}^{C_{\F}}(V_{\F})$.
The bottom row is the $\F$-twisted action of~$B$ on $\ind_B^C(V)$. To make the diagram more compact, we write 
$\F_{12}$ to denote $\F$ acting on the first two legs of the tensor product, i.e., the operator 
$(\F\rhd \otimes \id)$; similarly for other operators:
$$\xymatrix@C=2.5em@R=5em{B\otimes A\otimes V \ar[r]^{\F^{-1}_{12}} \ar[d]_{\F_{23}^{-1}} & B\otimes A\otimes V \ar[r]^{(\Psi_C)_{12}} \ar[d]_{(\triangle \otimes \id)(\F^{-1})\rhd} & A\otimes B \otimes V \ar[r]^{\F_{12}} \ar[d]_{(\triangle \otimes \id)(\F^{-1})\rhd} & A\otimes B\otimes V \ar[r]^{\F_{23}^{-1}} & A\otimes B\otimes V \ar[r]_{(\rhd_B)_{23}} \ar[d]_{(\id\otimes \triangle)(\F^{-1})\rhd} & A\otimes V \ar[d]_{\F^{-1}\rhd} \\
B\otimes A\otimes V \ar[r]_{(\id \otimes \triangle)(\F^{-1})\rhd } & B\otimes A \otimes V \ar[r]_{(\Psi_C)_{12}} & A\otimes B\otimes V \ar@{=}[r] & A\otimes B\otimes V \ar@{=}[r] & A\otimes B\otimes V \ar[r]_{(\rhd_B)_{23}} & A\otimes V}$$
The first (leftmost) square commutes by the cocycle equation, 
the second square commutes because $\Psi_C$ is an $H$-module morphism, the third square is the cocycle equation again, 
and the rightmost square commutes because the action $\rhd_B$ of $B$ 
is a morphism in $H\Mod$. Hence the diagram commutes, and 
the two actions of $B_\F$ are indeed intertwined by $\F^{-1}\rhd$,
as claimed.
%
%
%
\end{proof}

\section{Twisting irreducible characters of Coxeter groups}\label{twisting_characters_sec} 
In \cite[Theorem 6.3]{twistsrcas} we showed Drinfeld twists induce a non-trivial permutation of the four linear characters of the Coxeter group of type $B_n$. Here we will extend this result, showing how twisting permutes all of the irreducible characters of $B_n$. Additionally, we show how twisting induces a bijective correspondence between the irreducible characters of the Coxeter group of type $D_n$ and those of its mystic partner $\mu(D_n)$. In particular we describe the permutation for $B_n$, and the bijective correspondence between $D_n$ and $\mu(D_n)$, using a partition conjugation action.\bb

\nt Before discussing the Coxeter groups $B_n$ and $D_n$ let us first recall the family of complex reflection groups $G(m,p,n)$, and their mystic counterparts $\mu(G(m,p,n))$. For $n\in \N$, let $\mathbb{G}_n$ and $\mathbb{S}_n$ denote the groups of $n\times n$-monomial, and permutation, matrices over $\Cc$ respectively. 
\begin{definition}\label{gmpn_defn} For $m,p,n\in \N$ with $p|m$, define the following groups:
\begin{itemize}
  \item $G(m,p,n)$ is the subgroup of $\mathbb{G}_n$ formed by matrices where all the non-zero entries are $m$-th roots of unity, and the product of all non-zero entries is an $\frac{m}{p}$-th root of unity.
  \item $\mu(G(m,p,n))$ is the subgroup of $\mathbb{G}_n$ with non-zero entries being $m$-th roots of unity, and such that the determinant of the matrix is an $\frac{m}{p}$-th root of unity.
\end{itemize}
\end{definition}
\nt We write $(\Cc ^\times)^n $  to denote the subgroup of all diagonal matrices in $\mathbb G_n$ and put $T(m,p,n)= (\Cc ^\times)^n\cap G(m,p,n)$. The diagonal matrix whose $(i,i)$-th entry is $\epsilon$ and the rest of diagonal entries are $1$ will be denoted by $t_i^{(\epsilon)}$. The elements $t_i^{(-1)}$ play a special role in the paper, and we abbreviate $t_i^{(-1)}$ to $t_i$.\bb

\nt Note that if $\frac{m}{p}$ is even, then $\mu(G(m,p,n))=G(m,p,n)$. Bazlov and Berenstein \cite{mystic_reflections} defined a family of algebra isomorphisms $J_c:\mathbb{C}\mathbb{G}_n\xrightarrow{\sim}\mathbb{C}\mathbb{G}_n$, which for certain choices of $c\in \Cc^\times$ and $m,p\in \N$, restrict to an isomorphism from $\Cc \mu(G(m,p,n))$ to $\Cc G(m,p,n)$. For $c\in \Cc^\times,\ t\in (\Cc^\times)^n,\ w\in \mathbb{S}_n$, the map $J_c$ is given by
\begin{equation}\label{j_map}
J_c(wt):=wt \prod_{\{i<j| w(i)>w(j)\}}\frac{1}{4}\big ((c+c^{-1})(1-t_i t_j)+(c-c^{-1}+2)t_i+(c^{-1}-c+2)t_j \big)
\end{equation}
When $w=1$, the set ${\{i<j| w(i)>w(j)\}}$ is empty, and so $J_c(t)=t\ \forall t\in (\Cc^\times)^n$. Therefore the interesting behaviour of $J_c$ comes from how it evaluates on the standard generators $s_i,\ 1\leq i < n$, of $\mathbb{S}_n$. We will be particularly interested in the cases where $c=1$ and $c=-\mathbf i=\sqrt{-1}$. These evaluate on $s_i$ as follows:
\begin{align}
J_1(s_i) & =\frac{1}{2}s_i(1+t_i+t_{i+1}-t_it_{i+1})\label{j_1_eq}\\
J_{-\mathbf i}(s_i)& = \frac{1}{2}s_i((1-\mathbf i)t_i+(1+\mathbf i)t_{i+1})\label{j_minus_i_eq}
\end{align}
Note that when $\frac{m}{p}$ is even, $J_1$ restricts to an isomorphism $\Cc \mu(G(m,p,n))\xrt{\sim} \Cc G(m,p,n)$. Later in Section \ref{application_to_rcas_sec} we will see that, when $\frac{m}{p}$ is even, $J_1$ in fact coincides precisely with the embedding map $\eta\phi$ of Theorem \ref{embedding_thm}, after restricting $\eta\phi$ to the subalgebra $\Cc \mu(G(m,p,n))$ of the negative braided Cherednik algebra $\underline{H}_{\underline{c}}(\mu(G(m,p,n)))$. Here $\phi$ is an isomorphism, constructed in \cite{twistsrcas}, between the $\underline{H}_{\underline{c}}(\mu(G(m,p,n)))$
and the twist $H_c(G(m,p,n))_{\mathcal F}$ of the rational Cherednik algebra, and $\eta$ is a new map which arises from the results
given in Section~\ref{sec:general_twists}.\bb

\nt Now, since $J_1=\eta\phi$ is an isomorphism, the pullback of a representation $\rho:\Cc G(m,p,n)\rt \End(V)$ along $J_1$ gives rise to a representation $\rho\circ J_1$ of the group $\mu(G(m,p,n))$. Based on Remark \ref{pullback_rmk}, pulling back $\rho$ along $\eta$ corresponds to twisting the representation of $\Cc G(m,p,n)$ (in the sense of Giaquinto and Zhang) into a representation $\Cc G(m,p,n)_{\mathcal F}$. Pulling back $\rho\circ \eta$ along the map $\phi$ (from \cite[Theorem 5.2]{twistsrcas}) allows us to interpret the twisted representation $\rho\circ \eta$ as another group representation, in particular of $\mu(G(m,p,n))$. Therefore we interpret the action of pulling back a representation of $G(m,p,n)$ along $J_1$ as the action on representations induced by twisting.\bb

\nt In Section \ref{twisting_characters_b_sec} we describe this action explicity for the groups $G(2,1,n)$, which are the Coxeter groups of type $B_n$. We then turn our attention to the map $J_{-\mathbf i}$ in \eqref{j_minus_i_eq}. Although this map doesn't directly arise from any twisting constructions, it turns out that pulling back the irreducible characters of $B_n$ via $J_{-\mathbf i}$, instead of $J_1$, induces the same permutation of characters. The advantage of $J_{-\mathbf i}$ however is that it restricts to an isomorphism $\Cc \mu(G(m,p,n))\xrt{\sim} \Cc G(m,p,n)$ when $\frac{m}{p}$ is odd and even. We can therefore use $J_{-\mathbf i}$ to map the irreducible characters of Coxeter group of type $D_n$, given by $G(2,2,n)$, to those of $\mu(G(2,2,n))$. We describe this mapping explicity in Section \ref{twisting_characters_d_sec}.
\subsection{Twisting irreducible characters of \texorpdfstring{$B_n$}{}.}\label{twisting_characters_b_sec}
Recall the group $G(2,1,n)$ is isomorphic to the Coxeter group of type $B_n$. This group is well-known  \cite[Section 1.6.3]{alma9920315464401631} to have irreducible characters labelled by bipartitions of $n$, i.e.\ ordered pairs $(\lambda,\mu)$ where $\lambda$ and $\mu$ are partitions of $a,b\in \N$ respectively, where $a+b=n$. Since $\frac{m}{p}$ is even for this group, the map $\eta\phi$ of Theorem \ref{embedding_thm} restricts to an isomorphism $\Cc \mu(G(2,1,n))\xrightarrow{\sim} \Cc G(2,1,n)$. It is clear from Definition \ref{gmpn_defn} that $\mu(G(2,1,n))=G(2,1,n)$, and therefore we see $\eta\phi$ gives rise to an automorphism of $\Cc G(2,1,n)$. So the pullback of an irreducible character of $B_n$ along $\eta\phi$ will be another irreducible character of $B_n$. The following result describes this permutation of the characters explicity.

\begin{proposition}\label{permutation_of_characters_b_prop}
$\chi_{(\lambda,\mu)}\circ \eta\phi=\chi_{(\lambda,\mu^*)}$ where $\lambda \vdash a$ and $\mu \vdash b$ with $a+b=n$.

\begin{proof}
By \cite[Section 5.5.4]{alma9920315464401631}, the irreducible characters of $B_n$ are given as $$\chi_{(\lambda,\mu)}=\text{Ind}^{B_n}_{B_a\times B_b}(\tilde \chi_\lambda \boxtimes (\epsilon'_b \otimes \tilde \chi_\mu))$$
where $\tilde \chi_\mu\in \text{Irr}(B_a)$ is the pullback of the irreducible character $\chi_\lambda\in S_a$ (see \cite[Definition 5.4.4]{alma9920315464401631})
along the projection map $B_a \rightarrow S_a$, $\boxtimes$ denotes outer tensor product, and $\epsilon'_b$ is the restriction to $B_b$ of the linear character $\epsilon'$ on $B_n$ which sends $t_i\mapsto -1$ and $s_i\mapsto +1$. Precomposing with $\eta\phi$ we find:
\begin{align*}
\chi_{(\lambda,\mu)}\circ \eta\phi & = \text{Ind}^{B_n}_{B_a\times B_b}(\tilde \chi_\lambda \boxtimes (\epsilon'_b \otimes \tilde \chi_\mu)) \circ \eta \phi\\
& = \text{Ind}^{B_n}_{B_a\times B_b}(\chi_{(\lambda,\varnothing)} \boxtimes (\epsilon'_b \otimes \chi_{(\mu,\varnothing)})) \circ \eta \phi\\
& = \text{Ind}^{B_n}_{B_a\times B_b}(\chi_{(\lambda,\varnothing)} \boxtimes \chi_{(\varnothing,\mu)})) \circ \eta \phi.
\end{align*}
The first equality just uses the definition of $\chi_{(\lambda,\mu)}$, whilst the second rewrites $\tilde \chi_\lambda$ as $\chi_{(\lambda,\varnothing)}$, i.e. via the unique bipartition of $a$ which characterises it. We do similarly for $\tilde \chi_\mu$. In the 3rd equality we apply \cite[Theorem 5.5.6(c)]{alma9920315464401631} which says $\epsilon'\otimes \chi_{(\lambda,\mu)}=\chi_{(\mu,\lambda)}$.\bb

\nt Now note that the functor which sends a character $\psi$ of $B_a\times B_b$ to the
character $\text{Ind}^{B_n}_{B_a\times B_b}(\psi)\circ \eta\phi$
is a composition of two functors, 
the induction functor $\text{Ind}^{B_n}_{B_a\times B_b}\colon \text{Rep}(B_a\times B_b) \to \text{Rep}(B_n)$
followed by the autofunctor $\text{Precomp}_{\eta\phi}$ of $\text{Rep}(B_n)$ given by precomposing a representation of 
$\Cc B_n$ with $\eta\phi$. Note that $\text{Precomp}_{\eta\phi}$
is a permutation of irreducible characters of $B_n$ which is 
involutive, because $(\eta\phi)^2=(J_1)^2=\id$.
Hence $\text{Precomp}_{\eta\phi}$ is a self-adjoint functor.
It follows, by Frobenius reciprocity, that $\text{Precomp}_{\eta\phi}\circ \text{Ind}^{B_n}_{B_a\times B_b}$ is adjoint 
to the functor $\text{Res}^{B_n}_{B_a\times B_b}\circ\text{Precomp}_{\eta\phi} \colon \text{Rep}(B_n) \to \text{Rep}(B_a\times B_b)$. But since $\eta\phi$ restricted to $\Cc B_a\subseteq \Cc B_n$ is an automorphism of the group algebra $\Cc B_a$, same for $B_b$, it is obvious that 
this functor is the same as 
$(\text{Precomp}_{\eta\phi|_{B_a}}\boxtimes \text{Precomp}_{\eta\phi|_{B_b}})\circ \text{Res}^{B_n}_{B_a\times B_b}$. 
Taking adjoints again, we conclude that
$$
\chi_{(\lambda,\mu)}\circ \eta\phi = \text{Ind}^{B_n}_{B_a\times B_b}((\chi_{(\lambda,\varnothing)} \circ \eta \phi|_{B_a}) \boxtimes (\chi_{(\varnothing,\mu)} \circ \eta \phi|_{B_b})).
$$
At this point we claim and prove the following,
\begin{equation}\label{action_on_linear_characters}
\chi_{(\lambda,\varnothing)}\circ \eta\phi|_{B_a}=\chi_{(\lambda,\varnothing)}\hspace{1cm}\chi_{(\varnothing,\mu)}\circ \eta\phi|_{B_b}=\chi_{(\varnothing,\mu^*)}
\end{equation}
The left hand identity follows since $\chi_{(\lambda,\varnothing)}=\chi_\lambda \circ \pi$ for $\pi:B_n\rightarrow S_n$ the standard projection map, and on noting that $\pi\circ \eta\phi=\pi$, we deduce $\chi_{(\lambda,\varnothing)}\circ \eta\phi=\chi_\lambda \circ \pi \circ \eta\phi=\chi_\lambda \circ \pi=\chi_{(\lambda,\varnothing)}$.\bb

\nt It remains to check the right hand identity. Firstly, $(\epsilon'_b\circ \eta\phi)(t_i) = \epsilon'_b(t_i) = -1 = \epsilon_b(t_i), \forall i=a+1,\dots,n$. Also, $(\epsilon'_b\circ \eta\phi)(s_i) = \epsilon'_b(s_i\cdot \frac12 (1+t_i+t_{i+1}-t_it_{i+1})) = 1\cdot \frac12 (1-1-1-1) = -1 = \epsilon_b(s_i), \forall i=a+1,\dots,n-1$. Since $\epsilon'_b \circ \eta\phi$ and $\epsilon_b$ are algebra homomorphisms $\Cc B_b \to \Cc$ which agree on generators, one has $\epsilon'_b \circ \eta\phi = \epsilon_b$. So $\chi_{(\varnothing,\mu)}\circ \eta\phi|_{B_b}= 
(\epsilon'_b\otimes \chi_{(\mu,\varnothing)}) \circ \eta\phi|_{B_b} = (\epsilon'_b \circ \eta\phi|_{B_b} ) \otimes (\chi_{(\mu , \varnothing)} \circ \eta\phi|_{B_b} ) = \epsilon_b \otimes \chi_{(\mu , \varnothing)} = \chi_{(\varnothing,\mu^*)}$, where the final equality again applies \cite[Theorem 5.5.6(c)]{alma9920315464401631} which says $\epsilon\otimes \chi_{(\lambda,\mu)}=\chi_{(\mu^*,\lambda^*)}$.\bb

\nt With these two identities we find,
\begin{align*}
\chi_{(\lambda,\mu)}\circ \eta\phi & =\text{Ind}^{B_n}_{B_a\times B_b}((\chi_{(\lambda,\varnothing)} \circ \eta \phi|_{B_a}) \boxtimes (\chi_{(\varnothing,\mu)} \circ \eta \phi|_{B_b}))\\
& = \text{Ind}^{B_n}_{B_a\times B_b}(\chi_{(\lambda,\varnothing)}\boxtimes \chi_{(\varnothing,\mu^*)})\\
& = \chi_{(\lambda,\mu^*)}
\end{align*}
as required.
\end{proof}
\end{proposition}

\nt The above result shows how $\eta\phi$ gives rise to permutation of the irreducible characters of the Coxeter group of type $B_n$. In the next section we consider the Coxeter groups of type $D_n$, which are isomorphic to the groups $G(2,2,n)$. Since $\frac{m}{p}$ is odd in this case, we cannot pull back the irreducible characters of $G(2,2,n)$ along $\eta\phi$. 
However, it was shown in \cite[proof of Theorem 2.8]{mystic_reflections} that the map $J_c$ in \eqref{j_map} with $c=\mathbf i$ restricts to an isomorphism $\Cc G(m,p,n)\xrightarrow{\sim} \Cc \mu(G(m,p,n))$ when $\frac{m}{p}$ is both even and odd. Since $(J_c)^{-1}=J_{c^{-1}}$, we find $J_{-\mathbf i}$ restricts to an isomorphism $\Cc \mu(G(m,p,n))\xrightarrow{\sim} \Cc G(m,p,n)$. We can therefore ask how precomposition with $J_{-\mathbf i}$ acts on the irreducible characters of both $B_n$ and $D_n$. It follows from the next result that $J_{-\mathbf i}$ permutes the characters of the $B_n$ in exactly the same way as $J_1$ (or $\eta\phi$).
\begin{proposition}\label{j_i_j_1_prop}
$\chi_{(\lambda,\mu)}\circ J_1=\chi_{(\lambda,\mu)}\circ J_{-\mathbf i}$.
\begin{proof}
This follows by almost exactly the same argument as used in the proof of Proposition \ref{permutation_of_characters_b_prop},  with $\eta\phi$ replaced with $J_{-\mathbf i}$. The only part of the proof which does not immediately follow also for $J_{-\mathbf i}$ is the version of equations \eqref{action_on_linear_characters} with $J_{-\mathbf i}$ in place of $\eta\phi$.
But we see these hold too since, for the first equation, $(\pi \circ J_{-\mathbf i})(t)=\pi(t)$ and $(\pi\circ J_{-\mathbf i})(s_i)=\pi(\frac{1}{2}s_i((1-\mathbf i)t_i+(1+\mathbf  i)t_{i+1}))=\frac{1}{2}((1-\mathbf i)s_i+(1+\mathbf i)s_i)=s_i=\pi(s_i)$, so $\pi\circ J_{-\mathbf i}=\pi$ as required. For the second equation: $(\epsilon'_b\circ J_{-\mathbf i})(t_i)=\epsilon'_b(t_i)=-1=\epsilon_b(t_i)\ \forall i=a+1,\dots,n$ and also, $(\epsilon'_b\circ J_{-\mathbf i})(s_i)=\epsilon'_b(\frac{1}{2}s_i((1-\mathbf i)t_i+(1+\mathbf i)t_{i+1}))=-1=\epsilon_b(s_i)\ \forall i=a+1,\dots,n-1$.
\end{proof}
\end{proposition}
\nt What is the reason why 
$J_{-\mathbf i}$ acts the same way on irreducible characters of $B_n$ as $J_1=\eta\phi$? 
Using the fact $(J_1)^2=\id$, we can reframe Proposition \ref{j_i_j_1_prop} as saying
\begin{equation}\label{trivial_action_on_characters}
  \chi_{(\lambda,\mu)}\circ (J_{-\mathbf i}\circ J_1)=\chi_{(\lambda,\mu)}
\end{equation}
So the characters of $B_n$ are invariant under $J_{-\mathbf i}\circ J_1$. This means that $J_{-\mathbf i}\circ J_1$ is an inner automorphism of $\Cc G(2,1,n)$:
\begin{lemma}
    Suppose that $G$ is a finite group and $J\colon \Cc G\to \Cc G$ is an
    automorphism of its group algebra such that $\chi\circ J = \chi$ 
    for all irreducible characters $\chi$ of $G$. Then 
    there exists invertible $X\in \Cc G$ such that 
    $J(u)=XuX^{-1}$ for all $u\in \Cc G$.
\end{lemma}
\begin{proof}
Note that the condition $\chi\circ J = \chi$ means that for all
finite-dimensional $\Cc G$-modules $V$ and for 
all $u\in \Cc G$, the trace $\tr_V(J(u))$ is equal to the trace 
$\tr_V(u)$.\bb

\nt The centre $Z(\Cc G)$  of $\Cc G$ is spanned by primitive idempotents
$e_\chi$ labelled by irreducible characters $\chi$ of $G$.
As an automorphism, $J$ must permute the set $\{e_\chi\}$. 
Suppose that $J(e_\chi)=e_\psi$, and let $V_\chi$ be a $\Cc G$-module which affords $\chi$. Then 
$\tr_{V_\chi}(e_\psi)=\tr_{V_\chi}(e_\chi)=\dim V_\chi$.
However, if $\psi\ne \chi$ then $e_\psi$ acts on $V_{\chi}$ by zero.
Hence $\psi=\chi$, meaning that $J|_{Z(\Cc G)}$ is the identity map.\bb

\nt But then $J$ fixes each component in the 
decomposition $\Cc G = \prod_\chi e_\chi \Cc G$ of $\Cc G$
into the direct product of central simple algebras.
By the Skolem-Noether theorem, $J|_{e_\chi \Cc G}$
is given by $u\mapsto X_\chi u X_\chi^{-1}$ for all $\chi$.
Put $X=\prod_\chi X_\chi$.
\end{proof}


\nt We can conjecture from this that $J_1$ and $J_{-\mathbf i}$ twist the characters of all the groups $G(m,p,n)$, for $\frac{m}{p}$ even, in the same way.
\begin{conj}\label{twisting_char_gmpn_cor} If $\frac{m}{p}$ is even and $\chi$ is a character of $G(m,p,n)$, then $\chi\circ J_1=\chi\circ J_{-\mathbf i}$.
%
\end{conj}

\subsection{Twisting irreducible characters of \texorpdfstring{$D_n$}{}.}\label{twisting_characters_d_sec}
Having addressed the Coxeter group of type $B_n$, we turn to two of its index $2$ normal subgroups, the Coxeter group of type $D_n$ and its mystic counterpart $\mu(D_n)$, given by $G(2,2,n)$ and $\mu(G(2,2,n))$ respectively. We will use the map $J_{-\mathbf i}$ to twist the characters of $D_n$ into those of $\mu(D_n)$. Let us recall how the irreducible characters of these groups relate to those of $B_n$.\bb

\nt The irreducibles of $B_n$ are indexed by bipartitions $(\lambda,\mu)$, and, by \cite[Corollary 6.19]{alma992976864702801631}, the restriction to $D_n$ of an irreducible character of $B_n$ is either irreducible, or the sum of two distinct irreducibles. More specifically, for an irreducible $V_{(\lambda,\mu)}$ of $B_n$, if $\lambda\neq \mu$, then the restriction to $D_n$ remains irreducible. Whereas when $\lambda=\mu$, the restriction is a direct sum of two non-isomorphic irreducibles of $D_n$.\bb 

\nt Now $\mu(D_n)$ is also an index $2$ subgroup of $B_n$, so again by \cite[Corollary 6.19]{alma992976864702801631} the restriction to $\mu(D_n)$ of an irreducible character of $B_n$ will either be an irreducible, or the sum of two distinct irreducibles. In particular, $\mu(D_n)$ is the kernel of the linear character $\epsilon:B_n\rightarrow \Cc^\times,\ t\mapsto -1,s\mapsto -1$, and 
$\epsilon\otimes V_{(\lambda,\mu)}=V_{(\mu^*,\lambda^*)}$. We therefore find that, for $\lambda\neq \mu^*$, the irreducible $V_{(\lambda,\mu)}$ of $B_n$ restricts to an irreducible of $\mu(D_n)$. Whilst for $\lambda=\mu^*$, the restriction decomposes as a direct sum of two non-isomorphic irreducibles of $\mu(D_n)$.\bb

\nt In the following, let $\chi_{(\lambda,\mu)}$ denote the character of the irreducible representation $V_{(\lambda,\mu)}$ of $B_n$. When $\lambda\neq \mu$, let $\chi^{D_n}_{(\lambda,\mu)}$ denote the irreducible character arising from the restriction of $\chi_{(\lambda,\mu)}$ to $D_n$, i.e. $\res^{D_n}(\chi_{(\lambda,\mu)})$. Likewise, if $\lambda\neq \mu^*$, let $\chi^{\mu(D_n)}_{(\lambda,\mu)}:=\res^{\mu(D_n)}(\chi_{(\lambda,\mu)})$.

\begin{proposition}\label{perm_char_d_prop} $J_{-\mathbf i}$ induces a bijection between the irreducible characters of $D_n$ and the irreducible characters of $\mu(D_n)$. In particular, for $\lambda\neq \mu$, $\chi^{D_n}_{(\lambda,\mu)}\mapsto \chi^{\mu(D_n)}_{(\lambda,\mu^*)}$.
\begin{proof}
Precomposition with $J_{-\mathbf i}$ indeed gives a bijection of characters since $J_{-\mathbf i}$ restricts to an isomorphism $\Cc \mu(D_n)\rt \Cc D_n$. For brevity let us identify $J_{-\mathbf i}$ with this restriction. We then have,
$$\chi^D_{(\lambda,\mu)}\circ J_{-\mathbf i}=\res^{D_n}(\chi_{(\lambda,\mu)})\circ J_{-\mathbf i}=\res^{\mu(D_n)}(\chi_{(\lambda,\mu)}\circ J_{-\mathbf i})=\res^{\mu(D_n)}(\chi_{(\lambda,\mu^*)})$$
where the third equality applies Proposition \ref{permutation_of_characters_b_prop} and Proposition \ref{j_i_j_1_prop}. Since $\lambda\neq \mu$, we have $\lambda\neq (\mu^*)^*$, and therefore restricting $\chi_{(\lambda,\mu^*)}$ onto $\mu(D_n)$ gives an irreducible character, in particular $\chi^{\mu(D_n)}_{(\lambda,\mu^*)}$.
\end{proof}
\end{proposition}


\section{Rational Cherednik algebras}
Let us recall the result \cite[Theorem 5.2]{twistsrcas}, which showed that certain rational and negative braided Cherednik algebras are related by a Drinfeld twist. For $n\in \N$, let $V$ be an $n$-dimensional $\Cc$-vector space with basis $x_1,\dots,x_n$, and also let $y_1,\dots,y_n$ be the dual basis for $V^*$. For $1\leq i,j,k\leq n$ and $\epsilon\in \Cc^\times$ define the following maps in $\End(V)$:
\begin{equation*}
s_{ij}^{(\epsilon)}(x_k):=
\begin{cases}
x_k, & k\neq i,j,\\
\epsilon^{-1} x_j, & k=i,\\
\epsilon x_i, & k=j.
\end{cases}
\qquad 
t_i^{(\epsilon)}(x_k):=\epsilon^{\delta_{ik}}x_k,
\qquad
\sigma_{ij}^{(\epsilon)}(x_k):=\begin{cases}
x_k & k\neq i,j,\\
\epsilon^{-1} x_j & k=i,\\
-\epsilon x_i & k=j.
\end{cases}
\end{equation*}
Fix parameters $t\in \Cc$ and $c=\{c_1,\ c_\zeta\in \Cc \ |\ \zeta\in  C_{\frac{m}{p}}\backslash \{1\}\}$ and $c'=\{c'_1,\ c'_\zeta\in \Cc\ |\ \zeta\in C_{\frac{m}{p}}\backslash \{1\}\}$.
\begin{definition}\label{rational_cherednik_defn}
\begin{itemize}
  \item The \define{rational Cherednik algebra} $H_{t,c}(G(m,p,n))$ is the $\Cc$-algebra generated by $x_1,\dots,x_n\in V$, $y_1,\dots,y_n\in V^*$, and $g\in G(m,p,n)$, subject to the relations:
  \begin{align*}
  x_ix_j & =x_jx_i & gx_i & =g(x_i)g & y_ix_j - x_jy_i & =c_1 \sum_{\epsilon\in C_m}\epsilon s^{(\epsilon)}_{ij}\\
  y_iy_j & =y_jy_i & gy_i & =g(y_i)g & y_ix_i - x_iy_i & =t-c_1 \sum_{j\neq i}\sum_{\epsilon\in C_m}s^{(\epsilon)}_{ij}-\sum_{\zeta\in C_{\frac{m}{p}}\backslash \{1\}}c_\zeta t_i^{(\zeta)}
  \end{align*}
  for $1\leq i,j\leq n$ with $i\neq j$.

  \item The \define{negative braided Cherednik algebra} $\underline{H}_{t,c'}(\mu(G(m,p,n)))$ is the $\Cc$-algebra generated by $x_1,\dots,x_n\in V$, $y_1,\dots,y_n\in V^*$, and $g\in \mu(G(m,p,n))$, subject to the relations:
  \begin{align*}
  x_ix_j & =-x_jx_i & gx_i & =g(x_i)g & y_ix_j + x_jy_i & =c'_1 \sum_{\epsilon\in C_m}\epsilon \sigma^{(\epsilon)}_{ij}\\
  y_iy_j & =-y_jy_i & gy_i & =g(y_i)g & y_ix_i - x_iy_i& =t+c'_1 \sum_{j\neq i}\sum_{\epsilon\in C_m}\sigma^{(\epsilon)}_{ij}+\sum_{\zeta\in C_{\frac{m}{p}}\backslash \{1\}}c'_\zeta t_i^{(\zeta)}
  \end{align*}
  for $1\leq i,j\leq n$ with $i\neq j$.
\end{itemize}
\end{definition}

\nt Let $T:=\langle t_1,\dots,t_n\ |\ t_it_j=t_jt_i,\ t_i^2=1\rangle$, isomorphic to the $n$-fold direct product of the cyclic group of order $2$. By \cite[Proposition 5.1]{twistsrcas}, for $m$ even, $H_c(G(m,p,n))$ is a $\Cc T$-module algebra under the following action:
$$t_i\rhd g=t_ig t_i,\hspace{1em}t_i\rhd x_j=(-1)^{\delta_{ij}}x_j,\hspace{1em}t_i\rhd y_j=(-1)^{\delta_{ij}}y_j$$
for $1\leq i,j\leq n$ and $g\in G(m,p,n)$. Also recall from \cite[Lemma 4.5]{twistsrcas} that the following is a counital $2$-cocycle of $\Cc T$,
\begin{equation}\label{cocycle_f}
\F=\F^{-1}=\prod_{1\leq j< i\leq n}f_{ij}\text{, where }f_{ij}=\frac{1}{2}(1\otimes 1 +\gamma_i\otimes 1+1\otimes \gamma_j-\gamma_i\otimes \gamma_j)
\end{equation} 
and by \cite[Theorem 5.2]{twistsrcas}, we have $\underline{H}_{\underline{c}}(\mu(G(m,p,n)))\xrt{\sim} H_{c'}(G(m,p,n))_\F$ where $c'_1=-\underline{c}_1$ and $c'_\zeta = -\underline{c}_\zeta\ \forall \zeta\in C_{\frac{m}{p}}\backslash \{1\}$. Below in Theorem \ref{embedding_thm} we will denote the Drinfeld twist $H_{c'}(G(m,p,n))_\F$ by $(H_{c'}(G(m,p,n)),\star)$, where $\star$ is the twisted product of $H_{c'}(G(m,p,n))$.\bb

\nt Note that when $\frac{m}{p}$ is even, we have a natural embedding $u:\Cc T\rightarrow H_c(G(m,p,n))$, and the action of $\Cc T$ on $H_c(G(m,p,n))$ is seen to be adjoint with respect to this embedding. Therefore, by Proposition \ref{twist_isom_prop}(2), $H_{c'}(G(m,p,n))\cong H_{c'}(G(m,p,n))_\F$, and the module categories of these algebras must therefore be equivalent. In Theorem \ref{embedding_thm} we compute the isomorphism $H_{c'}(G(m,p,n))\cong H_{c'}(G(m,p,n))_\F$ explicitly, and compose it with the isomorphism of \cite[Theorem 5.2]{twistsrcas} in order to establish an explicit isomorphism between the negative braided Cherednik algebra $\underline{H}_{\underline{c}}(\mu(G(m,p,n)))$ and the rational Cherednik algebra $H_c(G(m,p,n))$. Note this isomorphism becomes an embedding when $\frac{m}{p}$ is odd.

\subsection{Application to Cherednik algebras.}\label{application_to_rcas_sec}

\begin{theorem}\label{embedding_thm}
Consider the negative braided Cherednik algebra $\underline{H}_{\underline{c}}(\mu(G(m,p,n))$ 
where $m$ is even and $n\geq 2$. Define $p'$, $c_1$ and $c_\zeta$ as follows,
\begin{align*}
p' =\begin{cases} p & \text{if } \frac mp\text{ is even,} \\
\frac p2& \text{if } \frac mp\text{ is odd,}
\end{cases} & \hspace{2em} 
c_1 = -\underline{c}_1, \hspace{2em}
c_\zeta =\begin{cases}-\underline{c}_\zeta & \text{if } \zeta \in C_{\frac{m}{p}}\setminus\{1\}\\ 0 & \text{if } \zeta\in C_{\frac{m}{p'}}\setminus C_{\frac{m}{p}}.\end{cases}
\end{align*}
Then $\underline{H}_{\underline{c}}(\mu(G(m,p,n))$ 
embeds inside the rational Cherednik algebra $H_c(G(m,p',n))$ via the following mapping of generators:
\begin{align*}
\underline{x}_i & \mapsto x_i t_{i-1}t_{i-2}\dots t_1,\\
\underline{y}_i & \mapsto y_i t_{i-1}t_{i-2}\dots t_1,\\
\sigma_i & \mapsto \frac{1}{2}(s_i+\bar s_i+\sigma_i-\sigma^{-1}_i)\\
t & \mapsto t\ \forall t\in T(m,p,n)
\end{align*}

\begin{proof}
Using the isomorphism $\phi$ from \cite[Theorem 5.2]{twistsrcas} we have an isomorphism $\underline{H}_{\underline{c}}(\mu(G(m,p,n)))\xrt{\sim} (H_{c'}(G(m,p,n)),\star)$ where $c'_1=-\underline{c}_1$ and $c'_\zeta = -\underline{c}_\zeta\ \forall \zeta\in C_{\frac{m}{p}}\backslash \{1\}$. Note that the rational Cherednik algebra $H_{c'}(G(m,p,n))$ is a subalgebra of $H_c(G(m,\frac{p}{2},n))$, with $c_1=c'_1$ and 
$$c_\zeta:=\begin{cases}c'_\zeta & \text{if } \zeta \in C_{\frac{m}{p}}\\ 0 & \text{if } \zeta\in C_{\frac{2m}{p}}\backslash C_{\frac{m}{p}}\end{cases}$$
Additionally for $u:\Cc T\rt H_c(G(m,\frac{p}{2},n))$ taken to be the natural embedding, $\Cc T$ acts adjointly (in the sense of Definition \ref{adjoint_action}) on $H_c(G(m,\frac{p}{2},n))$. On restricting to the subalgebra $H_{c'}(G(m,p,n))$ this action coincides with the action given in \cite[Proposition 5.1]{twistsrcas} which is used for twisting $H_{c'}(G(m,p,n))$. Therefore the twisted algebra $(H_{c'}(G(m,p,n),\star)$ is a subalgebra of twisted algebra $(H_c(G(m,\frac{p}{2},n)),\star)$.\bb

\nt Using the fact that the action of $\Cc T$ on $H_c(G(m,\frac{p}{2},n))$ is adjoint, we can apply Proposition \ref{twist_isom_prop}(1) to find $(H_c(G(m,\frac{p}{2},n)),\star)\cong H_c(G(m,\frac{p}{2},n))$. We deduce that $\underline{H}_{\underline{c}}(\mu(G(m,p,n))$ is isomorphic to a subalgebra of $H_c(G(m,\frac{p}{2},n))$ as required.\bb

\nt Next we inspect how the generators of $\underline{H}_{\underline{c}}(\mu(G(m,p,n))$ are mapped into $H_c(G(m,\frac{p}{2},n))$. Recall that under $\phi$, $\sigma_i\mapsto \bar s_i,\ t\mapsto t,\ \underline{x}_i\mapsto x_i,\ \underline{y}_i\mapsto y_i$. It remains to show how the map $\eta$ defined in Proposition \ref{twist_isom_prop} maps $\bar s_i, t, x_i$ and $y_i$. Let us use the cocycle given in \eqref{cocycle_f}. Let $\eta_{f_{ij}}(a):=\sum_{k,l=0}^1 (-1)^{kl}(\gamma_i^k\rhd a)u(\gamma_j^l)$. Then, using the commutativity of $\Cc T$, we see $\eta(a)=\eta_{f_{2,1}}\circ \eta_{f_{3,1}}\circ \dots \circ \eta_{f_{n,n-1}}$, i.e. the composition (in any order) of $\eta_{f_{ij}}$ for $1\leq j<i\leq n$. Now $u(\gamma_i)=t_i$, and it is easy to verify that 
$$\eta_{f_{ij}}(x_k)=\begin{cases}x_k & \text{if } i\neq k\\
x_k t_j & \text{if }i=k\end{cases}$$
Therefore, 
\begin{align*}
\eta(x_k) & =\eta_{f_{k,1}}\circ \eta_{f_{k,2}}\circ \dots \eta_{f_{k,k-1}}(x_k)\\
 & =\eta_{f_{k,1}}\circ \eta_{f_{k,2}}\circ \dots \eta_{f_{k,k-2}}(x_k t_{k-1})\\
 & \dots\\
 & = x_k t_{k-1} \dots t_1
\end{align*}
One similarly shows that $\eta(y_k)=y_k t_{k-1}\dots t_1$. Also, since $\Cc T(m,\frac{p}{2},n)$ is commutative, the adjoint action of the subalgebra $\Cc T$ on $\Cc T(m,\frac{p}{2},n)$ is trivial, so $\eta_{f_{ij}}(t)=\frac{1}{2}(t\cdot 1+t\cdot 1 +t \cdot t_j-t\cdot t_j)=t$ and $\eta(t)=t$. Finally we check $\eta(\bar s_i)$. Using $p_j+q_j=1$ we find
$$\eta_{f_{ij}}(\bar s_k)=\bar s_k p_j+(\gamma_i\rhd \bar s_k)q_j=\begin{cases}\bar s_k & \text{if } i\neq k,k+1\\
\bar s_k p_j+s_k q_j & \text{if } i=k,k+1\end{cases}$$
Additionally, 
$$\eta_{f_{ij}}(s_k)=\begin{cases}s_k & \text{if } i\neq k,k+1\\
s_k p_j+\bar s_k q_j & \text{if } i=k,k+1
\end{cases}$$
So $(\eta_{f_{k+1,j}}\circ \eta_{f_{k,j}})(\bar s_k)=\bar s_k$ (using the fact $p_j$ and $q_j$ are orthogonal idempotents). Therefore $\eta(\bar s_k)=\eta_{f_{k+1,k}}(\bar s_k)=\bar s_k p_k+s_k q_k=\frac{1}{2}(s_k + \bar s_k +\sigma_k-\sigma^{-1}_k)$.
\end{proof}
\end{theorem}

\nt Since the embedding map of Theorem \ref{embedding_thm} is an isomorphism when $\frac{m}{p}$ is even, we deduce:

\begin{cor} For $\frac{m}{p}$ even, the category of modules over the negative braided Cherednik algebra $\underline{H}_{\underline{c}}(\mu(G(m,p,n))$ is equivalent to the category of modules over the rational Cherednik algebra $H_c(G(m,p,n))$.
\end{cor}

\subsection{Twisting standard modules of rational Cherednik algebras.}
Standard modules for rational Cherednik algebras are modelled after induced modules introduced by Verma in \cite{verma_1966} and further studied by Bernstein, Gelfand and Gelfand in \cite{BGG_1971}. Here we define a natural analogue of standard modules for negative braided Cherednik algebras, and show that the twist (in the sense of  \eqref{eq:def_twist}) of a standard module of a rational Cherednik algebra is precisely one of these standard modules for the corresponding negative braided Cherednik algebra.\bb

\nt For $W$ an irreducible complex reflection group with reflection representation $V$, let $\HH=H_c(W)$ be the associated rational Cherednik algebra (with $t=1$).

\begin{definition}\label{standard_module} Let $\tau$ be a simple $\Cc W$-module, and extend to a $S(V)\rtimes W$-module in which $V$ acts by zero, i.e. for $p\in S(V),v\in \tau$, $p\rhd v:=p(0)v$. Then the \define{standard module} associated to $\tau$ is the $H_c(W)$-module $M_c(\tau):=H_c(W)\otimes_{S(V)\rtimes W}\tau$.
\end{definition}

\nt In the following we suppose $m$ is even, and let $W=G(m,p,n)$ and $\mu(W)=\mu(G(m,p,n))$. Let $\underline{\HH}=\underline{H}_{\underline c}(\mu(W))$ be the negative braided Cherednik algebra isomorphic to the twist of $\HH$ under the map $\phi$ from \cite[Theorem 5.2]{twistsrcas}. Note that $\underline{\HH}$ contains the subalgebra $S_{-1}(V)\rtimes \mu(W)$, and $\phi:\underline{\HH}\xrightarrow{\sim} \HH_\F$ restricts to an isomorphism $S_{-1}(V)\rtimes \mu(W)\xrt{\sim} (S(V)\rtimes W)_\F$. We have also have the following natural analogue of Definition \ref{standard_module}: for $\tau'$ a simple $\Cc \mu(W)$-module, extend to a $S_{-1}(V)\rtimes \mu(W)$-module by letting $V$ act on $\tau'$ by $0$. Then define the \define{standard $\underline{\HH}$-module} as $\underline{M}_{\underline c}(\tau'):=\underline{\HH}\otimes_{S_{-1}(V)\rtimes \mu(W)}\tau'$.\bb

\nt In the following we will prove several cases where the standard $\underline{\HH}$-module is (up to an isomorphism of categories) isomorphic to the Giaquinto and Zhang twist of the standard $\HH$-module, $M_c(\tau)_\F$. To prove this we first require the following elementary result.\bb

\nt Suppose $\phi:B\xrt{\sim}B'$ is an algebra isomorphism, $A$ is a subalgebra of $B$, and $A'=\phi(A)$. $\phi$ defines an isomorphism of categories $B'\Mod\xrt{\sim} B\Mod$, whereby a $B'$-module $(V,\ \rhd:B'\otimes V\rt V)$ is sent to the $B$-module $(V^\phi:=V,\ \rhd^\phi:=\rhd\circ (\phi\otimes \id_V))$. Note that the map $\phi|_A:A\xrt{\sim} A'$ also defines an isomorphism of categories $A'\Mod\xrt{\sim}A\Mod$. The image of an $A'$-module $(V,\ \rhd:A'\otimes V\rt V)$ under this functor will be similarly denoted $(V^\phi:=V,\ \rhd^\phi=\rhd \circ (\phi|_A\otimes \id_V))$.
\begin{lemma}\label{ind_isom_lem} If $V$ is an $A'$-module, then $(\ind_{A'}^{B'}(V))^\phi\cong \ind_A^B(V^\phi)$ as $B$-modules.
\begin{proof} We already proved a special case of this in the proof of Proposition \ref{permutation_of_characters_b_prop}. Note the functor $\text{Precomp}_{\phi}$ has an adjoint given by $\text{Precomp}_{\phi^{-1}}$. Also $\ind^B_A$ and $\ind^{B'}_{A'}$ are adjoint to $\res^B_A$ and $\res^{B'}_{A'}$ respectively. Therefore $\text{Precomp}_{\phi}\circ \ind^{B'}_{A'}$ has adjoint $\res^{B'}_{A'}\circ \text{Precomp}_{\phi^{-1}}$, which is easily seen to be equal to $\text{Precomp}_{\phi^{-1}|_{A'}}\circ \res^B_A$. The adjoint of this is then $\ind^B_A\circ \text{Precomp}_{\phi|_A}$, as required.
\end{proof}
\end{lemma}

\nt Recall the map $J_{-\mathbf i}$ of \eqref{j_minus_i_eq} restricts to an isomorphism $\Cc \mu(W)\xrightarrow{\sim} \Cc W$ (regardless of the parity of $\frac{m}{p}$). This map therefore induces an isomorphism of categories $\Cc W\Mod \xrightarrow{\sim} \Cc \mu(W)\Mod, \tau\mapsto \underline{\tau}$, where $\underline{\tau}=\tau$ as a $\Cc$-vector space, and if $\rhd:\Cc W\otimes \tau\rightarrow \tau$ is the action of $\Cc W$ on $\tau$, then the action of $\Cc \mu(W)$ on $\underline{\tau}$ is given by
$$\rhd\circ (J_{-\mathbf i}\otimes \id_{\underline\tau}):\Cc \mu(W)\otimes \underline\tau\rightarrow \underline\tau$$
In the above notation, $\underline{\tau}=\tau^{J_{-\mathbf i}}$. Similarly, the map $\phi:\underline{\HH}\xrightarrow{\sim} \HH_\F$ defines an isomorphism of categories $\HH_\F\Mod\xrt{\sim}\underline{\HH}\Mod$, and below we take $M_c(\tau)_\F^\phi$ to denote the image of $M_c(\tau)_\F$ under this functor.

\begin{proposition} 
Suppose that $\tau$ is either a simple $G(2,1,n)$-module, or a simple $G(2,2,n)$-module corresponding to a bipartition $(\lambda,\mu)$ of $n$ where $\lambda\neq \mu$ (see Section \ref{twisting_characters_d_sec}). Then $M_c(\tau)_\F^\phi \cong \underline{M}_{\underline{c}}(\underline \tau)$ as $\underline{\HH}$-modules.
\begin{proof}
In following we will denote $G(2,1,n)$ by $B$, $G(2,2,n)$ by $D$, and $\mu(G(2,2,n))$ by $\mu(D)$. Also let $W$ denote an arbitrary group $G(m,p,n)$, where $m$ is even. Note $M_c(\tau)=\text{Ind}_{S(V)\rtimes W}^\HH(\tau)$ and $\underline{M}_{\underline{c}}(\underline \tau)=\ind_{S_{-1}(V)\rtimes \mu(W)}^{\underline{\HH}}(\underline{\tau})$. \bb

\nt We wish to apply Proposition \ref{twisting_induced_rep_prop} to the Giaquinto-Zhang twist $M_c(\tau)_\F=(\text{Ind}_{S(V)\rtimes W}^\HH(\tau))_\F$. To apply the proposition we first must note that $\HH=S(V^*)\cdot (S(V)\rtimes W)$ is indeed an algebra factorisation in $\Cc T\Mod$, where recall $T$ is the group isomorphic to $(C_2)^n$. We also require $\tau \in (\Cc T,S(V)\rtimes W)\Mod$, which we check next. Note this category is well-defined since $S(V)\rtimes W$ is a $\Cc T$-module algebra under the adjoint action of $\Cc T$ (in particular it is a $\Cc T$-submodule algebra of $\HH$). When $\frac{m}{p}$ is even, then $\Cc T$ is a subalgebra of $\Cc W$, and therefore $\tau$ is naturally a $\Cc T$-module by restricting the $S(V)\rtimes W$ action to $\Cc T$. This action can equivalently be seen as pulling back the action of $S(V)\rtimes W$ along the embedding map $\Cc T\hookrightarrow S(V)\rtimes W$, and therefore we can apply Proposition \ref{twist_isom_prop}(3) to deduce $\tau \in (\Cc T,S(V)\rtimes W)\Mod$.\bb

\nt However $T$ does not embed in $D=G(2,2,n)$. By assumption though, $\tau$ is equal to the simple $\Cc D$-module $\tau^D_{(\lambda,\mu)}$, where $\lambda\neq \mu$. On recalling Section \ref{twisting_characters_d_sec}, this module is the restriction to $D$ of either of the simple $\Cc B$-modules $\tau^B_{(\lambda,\mu)}$ or $\tau^B_{(\mu,\lambda)}$. Since $\Cc T$ is a subalgebra of $\Cc B$, each of these $\Cc B$-modules defines an action of $\Cc T$ on $\tau$.   Note that on extending $\tau$ to one these $\Cc B$-modules, we can apply the $\frac{m}{p}$ even case above to deduce that $\tau\in (\Cc T, S(V)\rtimes B)\Mod$. But then the compatibility condition \eqref{eq:compatibility} will still hold if we restrict the action of $S(V)\rtimes B$ on $\tau$ to $S(V)\rtimes D$, and therefore $\tau\in (\Cc T, S(V)\rtimes D)\Mod$ as required. Later we show that this proposition holds independently of which action is chosen.\bb

\nt We can now apply Proposition \ref{twisting_induced_rep_prop} to deduce:
$$M_c(\tau)_\F\cong \text{Ind}_{(S(V)\rtimes W)_\F}^{\HH_\F}(\tau_\F)$$

\nt Here $\tau_\F$ is the Giaquinto and Zhang twist of $\tau$ (see \eqref{eq:def_twist}), so the action of $(S(V)\rtimes W)_\F$ on $\tau_\F$ is given by $\rhd_\F:=\rhd\circ (\F^{-1}\blacktriangleright\ -\ )$, where $\rhd$ denotes the action of $S(V)\rtimes W$ on $\tau$, and $\blacktriangleright$ denotes the action $\Cc T\otimes \Cc T$ on $(S(V)\rtimes W)\otimes \tau$. We now apply Lemma \ref{ind_isom_lem} to find,
$$M_c(\tau)^\phi_\F\cong (\text{Ind}_{(S(V)\rtimes W)_\F}^{\HH_\F}(\tau_\F))^\phi\cong \ind^{\underline{\HH}}_{S_{-1}(V)\rtimes \mu(W)}(\tau_\F^\phi)$$
where $\tau^\phi_\F=\tau$ as a $\Cc$-vector space, and has action of $S_{-1}(V)\rtimes \mu(W)$ given by $\rhd^\phi_\F=\rhd_\F \circ (\phi\otimes \id_\tau)$.\bb

\nt To complete the proof it remains to check $\tau^\phi_\F=\underline{\tau}$ as $S_{-1}(V)\rtimes \mu(W)$-modules. We consider the cases of $\frac{m}{p}$ even, and $G(2,2,n)$, separately.\bb

\nt Let $\frac{m}{p}$ be even. The action of $\Cc T$ on $\tau$ is just a restriction of the action of $S(V)\rtimes W$ on $\tau$, and the action of $\Cc T$ on $S(V)\rtimes W$ is adjoint, so we can apply Proposition \ref{twist_isom_prop}(3). We deduce that the action $\rhd_\F$ of $\tau_\F$ is given by the pulling back $\rhd$ along the map $\eta$, i.e. $\rhd_\F=\rhd\circ (\eta\otimes \id_\tau)$, where $\eta:(S(V)\rtimes W)_\F\xrt{\sim} S(V)\rtimes W$ was defined in the proof of Theorem \ref{embedding_thm}. Then $\rhd^\phi_\F=\rhd \circ (\eta\phi\otimes \id_\tau)$, where $\eta\phi$ is regarded here as a map $S_{-1}(V)\rtimes \mu(W)\xrt{\sim} S(V)\rtimes W$.\bb

\nt Note that the map $\eta\phi$ is degree-preserving, and, by definition, $\rhd$ is such that elements of degree $\geq 1$ in $S(V)\rtimes W$ act on $\tau$ by $0$. Therefore $\rhd^\phi_\F$ is determined by how it behaves on the subalgebra $\Cc \mu(W)$ of $S_{-1}(V)\rtimes \mu(W)$. Recall that the restriction of the map $\eta\phi$ to the subalgebra $\Cc\mu(W)$ coincides precisely with the map $J_1$, given in \eqref{j_1_eq}. Also, in Proposition \ref{j_i_j_1_prop} it was proven that $\chi\circ J_1=\chi\circ J_{-\mathbf i}$, for $\chi$ an irreducible character of $B_n$. Therefore, as actions of $\Cc\mu(W)$ on $\tau$, we have $\rhd\circ (J_1\otimes \id_\tau)=\rhd\circ (J_{-\mathbf i}\otimes \id_\tau)$. Hence, as $\Cc \mu(W)$-modules, $\tau^\phi_\F =\underline{\tau}$. This implies they are also equal as $S_{-1}(V)\rtimes \mu(W)$-modules, since $V$ acts by $0$ for both $\tau^\phi_\F$ and $\underline{\tau}$. This proves the result for the group $B=G(2,1,n)$.\bb 

\nt Let us now check $\tau^\phi_\F=\underline{\tau}$, as $S_{-1}(V)\rtimes \mu(D)$-modules, when $\tau=\tau^D_{(\lambda,\mu)}$ for $\lambda\neq \mu$. As mentioned before, $\tau$ can be extended to either of the simple $\Cc B$-modules $\tau^B_{(\lambda,\mu)}$ or $\tau^B_{(\mu,\lambda)}$. Suppose, for now, that we extend $\tau$ to $\tau^B_{(\lambda,\mu)}$. Let us consider $(\tau^B_{(\lambda,\mu)})^\phi_\F$. From the $\frac{m}{p}$ even case above we know that the action of $\Cc B$ on $(\tau^B_{(\lambda,\mu)})^\phi_\F$ is given by $\rhd \circ (\eta\phi|_{\Cc B}\otimes \id_\tau)=\rhd\circ (J_1\otimes \id_\tau)$. Therefore by Proposition \ref{permutation_of_characters_b_prop}, $(\tau^B_{(\lambda,\mu)})^\phi_\F=\tau^B_{(\lambda,\mu^*)}$. Since $\lambda\neq (\mu^*)^*$, restricting the $\Cc B$-action on $\tau^B_{(\lambda,\mu^*)}$ to $\Cc\mu(D)$ gives the irreducible representation $\tau^{\mu(D)}_{(\lambda,\mu^*)}$. Let us now inspect $\underline{\tau}$, which is defined to be $(\tau^D_{(\lambda,\mu)})^{J_{-\mathbf i}}$. By Proposition \ref{perm_char_d_prop}, $(\tau^D_{(\lambda,\mu)})^{J_{-\mathbf i}}=\tau^{\mu(D)}_{(\lambda,\mu^*)}$. Therefore $\tau^\phi_\F=\underline{\tau}$ as $\Cc \mu(D)$-modules, and so also as $S_{-1}(V)\rtimes \mu(D)$-modules, since each of these $\Cc \mu(D)$-modules is extended by taking $V$ to act by $0$.\bb 

\nt Finally suppose we had instead extended $\tau^D_{(\lambda,\mu)}$ to $\tau^B_{(\mu,\lambda)}$. Then on applying Proposition \ref{permutation_of_characters_b_prop} we would have found $(\tau^B_{(\mu,\lambda)})^\phi_\F=\tau^B_{(\mu,\lambda^*)}$. Upon restricting this to $\mu(D)$ we arrive at the irreducible representation $\tau^{\mu(D)}_{(\mu,\lambda^*)}$. But we see that $\tau^{\mu(D)}_{(\mu,\lambda^*)}=\tau^{\mu(D)}_{(\lambda,\mu^*)}$ since by \cite[Theorem 5.5.6(c)]{alma9920315464401631}, $\epsilon \otimes \tau^B_{(\lambda,\mu)}=\tau^B_{(\mu^*,\lambda^*)}$, where $\epsilon$ is the linear character which satisfies $\mu(D)=\ker(\epsilon)$.

\end{proof}

\end{proposition}

\section{Twisted coinvariant algebras}
\label{sec:twists_coinv_alg}

\subsection{Classical (commutative) coinvariant algebras.}

Let $V$ be a finite-dimensional vector space over
a field $\Kk$. Each 
finite subgroup $G$ of $\GL(V)$ acts on the symmetric $\Kk$-algebra $S(V)$ of $V$. Classically, invariant theory of finite groups studies the graded algebra $S(V)^G$ of $G$-invariants. Of particular interest is the case 
when $G$ is a finite reflection group, that is, when 
$\{s\in G: \rank(\id_V-s)=1\}$ generates $G$. 
In the non-modular case (when 
the characteristic of $\Kk$ does not divide $|G|$),
the work of Chevalley, Sheppard and Todd, and Serre in the 1950s
established that $G$ is a reflection group, if and only if 
$S(V)^G$ is a polynomial algebra. 
This is known as the Chevalley-Shephard-Todd theorem, see for example \cite[Theorem 7.1.4]{NeuselSmith2002}.\bb

\nt The \define{coinvariant algebra} $S_G$ is defined as the quotient algebra $S(V)/I_G$ where $I_G$ is the ideal of $S(V)$ generated by
homogeneous $G$-invariants of positive degree.
%
The following result was initially proved by Chevalley \cite{Chevalley1955} under more restrictive assumptions, but holds for all non-modular reflection groups, see \cite[Theorem 7.2.1]{NeuselSmith2002}. See also \cite{Broer2011} for a discussion of the modular case.
\begin{theorem}\label{thm:chevalley}
Supposing $k$ is an arbitrary field and $G$ is a finite reflection group generated by reflections whose orders are prime to the characteristic of $k$. Then the coinvariant algebra, viewed as a $G$-module, affords the regular representation of $G$. 
\end{theorem}


\subsection{The noncommutative coinvariant algebra of a mystic reflection group.}

Let $\Kk$ be a field of characteristic $0$. An extension of the Chevalley-Shephard-Todd theorem to skew-polynomial $\Kk$-algebras $S_{\mathbf q}(V)$ is 
given by Kirkman, Kuzmanovich and Zhang in \cite{Kirkman2008ShephardToddChevalleyTF}. It characterises mystic reflection groups as finite subgroups $W\subset \GL(V)$ which act by algebra automorphisms on $S_{\mathbf q}(V)$ such that the subalgebra $S_{\mathbf q}(V)^W$ is again skew-polynomial. 
Set $\Kk=\Cc$; it turns out that the ``building blocks'' of mystic reflection groups are
\begin{itemize}
\item the usual complex reflection groups, which act on the commutative polynomial algebra $S(V)$,
\item the family $\mu(G(m,p,n))$ of Definition~\ref{gmpn_defn} where $m$ is even, which act on the skew-polynomial algebra $S_{-1}(V)$.
\end{itemize}
Recall that 
the group $\mu(G(m,p,n))$ is defined via its faithful action on a space $V$ with a fixed basis $x_1,\dots,x_n$. 
We regard the graded vector space
$$
S = \bigoplus_{(k_1,\dots,k_n)\in \mathbb{Z}_{\ge 0}^n}
\Cc x_1^{k_1}\dots x_n^{k_n},
$$
spanned by the \define{standard monomials},
as the underlying space for both $S(V)$ and $S_{-1}(V)$. The space 
$S$ contains 
\begin{equation}\label{eq:invariant_generators}
p_k^{(m)} = \sum_{i=1}^n x_i^{km}, \quad k=1,\dots,n-1; \qquad r^{(\frac mp)}=x_1^{\frac mp} x_2^{\frac mp}  \dots x_n^{\frac mp}, 
\end{equation}
which are homogeneous algebraically independent generators,
in $S(V)$, of the subalgebra of $G(m,p,n)$-invariants, 
see \cite[Chapter 2, Section 8]{alma9930780234401631}.
The subalgebra of $\mu(G(m,p,n))$-invariants in $S_{-1}(V)$ happens to form a commutative algebra isomorphic to $S(V)^{G(m,p,n)}$, as noted in \cite{Kirkman2008ShephardToddChevalleyTF}.
\begin{theorem}[\cite{mystic_reflections}, Theorem 2.6]
Let $W=\mu(G(m,p,n))$. Then $S_{-1}(V)^W$ is a commutative polynomial algebra, freely generated by $p_1^{(m)},\dots,p_{n-1}^{(m)}$ and $r^{(\frac mp)}$.
\end{theorem}

\nt Given that the mystic reflection groups $\mu(G(m,p,n))$ have polynomial invariants in $S_{-1}(V)$, we define their coinvariant algebras and obtain the noncommutative analogue of Chevalley's classical result.\bb

\nt In the following, take $G=G(m,p,n)$ and $W=\mu(G(m,p,n))$.
%

\begin{definition}
Let $I_{W}$ be the two-sided ideal of the algebra $S_{-1}(V)$ generated
by homogeneous $W$-invariants of positive degree. The \define{noncommutative coinvariant algebra}
of $W$ is the quotient algebra $\underline S_{W} = S_{-1}(V)/I_W$.  
\end{definition}
\nt It is clear from the definition that the ideal $I_W$ is a $W$-submodule of $S_{-1}(V)$,
and so $W$ acts on $\underline S_W$ by algebra automorphisms --- the action descends from $S_{-1}(V)$.

\begin{theorem}\label{thm:nccoinv}
The noncommutative coinvariant algebra $\underline{S}_W$ of the mystic reflection group $W=\mu(G(m,p,n))$ affords the regular representation of $W$.
\end{theorem}
\nt
We prove Theorem~\ref{thm:nccoinv} over the next few sections, using the techniques we developed above for working with cocycle twists.

\subsection{\texorpdfstring{$W$}{}-actions and Giaquinto-Zhang twists of \texorpdfstring{$G$}{}-actions.} 

We are going to use the rational Cherednik algebra 
$H:=H_{0,0}(G)$ of $G=G(m,p,n)$ as well as the negative braided
Cherednik algebra $\underline{H} := \underline{H}_{0,0}(W)$.
We will make use only of the semidirect product relations in these algebras, so for simplicity we set the parameters $t,c$ to $0$.
To state the next Lemma, we need to recall from \cite[Remark 5.10]{twistsrcas} that the group algebra $\Cc W$ is the middle term in the PBW-type algebra factorisation
\begin{equation}\label{eq:pbw_braided}
\underline{H} = S_{-1}(V) \otimes \Cc W \otimes S_{-1}(V^*).
\end{equation}
The result \cite[Theorem 5.2]{twistsrcas}, establishes an algebra isomorphism $\phi\colon \underline H \to H_{\mathcal F}$ between $\underline H$ and the cocycle twist of $H$ by $\mathcal F$.
The map $\phi$ is defined on generators of $\underline H$ by 
$$
\phi(\sigma_i)=\bar s_i, \quad 
\phi(t)=t,\quad
\phi(x_j)=x_j, \quad \phi(y_j)=y_j
$$
for $i\in [n-1]$, $t\in T(m,p,n)$ and $j\in [n]$.
In particular, $\phi$ restricts to an isomorphism 
$\Cc W \to (\Cc G)_{\mathcal F}$. Furthermore,
\begin{equation} \label{eq:phi-identity}
\phi|_{S_{-1}(V)} = \id_S\colon S_{-1}(V) \to S(V)_{\mathcal F},
\end{equation}
i.e., $\phi$ is the identity on the underlying vector space 
$S$ because it maps every standard monomial 
in $x_1,\dots,x_n$ to itself \cite[section 5.9]{twistsrcas}.
We show that $\phi$ intertwines the actions of 
$\Cc W$ and $(\Cc G)_{\mathcal F}$ on $S$:
\begin{lemma} \label{lem:twisted_action}
	Let $\underline{\rhd}$ be the action of $\Cc W$ on $S_{-1}(V)$, and let
	$\rhd_{\mathcal F}\colon (\Cc G)_{\mathcal F}
	\otimes S(V)_{\mathcal F}\to S(V)_{\mathcal F}$ be the Giaquinto-Zhang twist by $\mathcal F$ of the natural action 
	 of $\Cc G$ on $S(V)$.
	Then for all $u\in \Cc W$, 
	$u\mathop{\underline{\rhd}} = \phi(u)\rhd_{\mathcal F}$ 
	as endomorphisms of the vector space $S$.
\end{lemma}
\begin{proof}
We bring in the Cherednik algeras.
The relations in Definition~\ref{rational_cherednik_defn} 
mean that the action 
$\underline{\rhd}\colon \Cc W \otimes S_{-1}(V)\to S_{-1}(V)$ can be written as the composite map
\begin{equation}
\label{eq:action_composite}	
\Cc W \otimes S_{-1}(V) \hookrightarrow 
\underline{H} \otimes \underline{H}
\xrightarrow{\underline m}
\underline{H} \xrightarrow{\id_{S_{-1}(V)} \otimes \epsilon_W \otimes \epsilon_{V^*}} S,
\end{equation}
where $\underline m$ is the product map on $\underline H$, $\epsilon_W\colon \Cc W \to \Cc$ is the augmentation map of the group algebra $\Cc W$, and $ \epsilon_{V^*}\colon 
S_{-1}(V^*) \to \Cc$ maps a polynomial to its constant term.
The domain of $\id_{S_{-1}(V)} \otimes \epsilon_W \otimes \epsilon_{V^*}$ is the triple tensor product in~\eqref{eq:pbw_braided}.
Thus, \eqref{eq:action_composite} says that the $\Cc W$-module
$S_{-1}(V)$ is the restriction to $\Cc W$ of the representation 
of $\underline H = S_{-1}(V)\cdot (S_{-1}(V^*)\rtimes W)$ induced from the one-dimensional trivial module $\epsilon_W\otimes \epsilon_{V^*}$  of the semidirect product $S_{-1}(V^*) \rtimes W$. Consider the diagram
$$
\xymatrix@C=5.5em@R=2.5em{ 
%
%
\Cc W \otimes S_{-1}(V)
	\ar[r]^{\subset} \ar[d]_{\phi\otimes \id} & 
\underline{H} \otimes \underline{H}	
	 \ar[r]^{\underline m} \ar[d]_{\phi\otimes\phi} & \underline H \ar[r]^{\id \otimes \epsilon_W \otimes \epsilon_{V^*}} \ar[d]_{\phi} & S \ar[d]^{\id} \\
	(\Cc G)_{\mathcal F} \otimes S(V)_{\mathcal F}\ar[r]^{\subset } & 
	H_{\mathcal F}\otimes H_{\mathcal F} 
	 \ar[r]^{m_{\mathcal F}} & H_{\mathcal F} \ar[r]
	 ^{\id \otimes \epsilon_G \otimes \epsilon_{V^*}}
	  & S }
$$
where the leftmost square commutes because $\phi$ restricts to the identity on $S_{-1}(V)$, and the middle square commutes because $\phi$ is an isomorphism of algebras.\bb

\nt
We claim that the rightmost square commutes. This boils down to the equality $e_G \phi = e_W$ of maps $\Cc W \to \Cc$, where 
$e_G\colon (\Cc G)_{\mathcal F}\to\Cc$ is the augmentation map
of $\Cc G$. This may not be obvious, since $\phi$ does not send group elements of $W$ to group elements of $G$. Note that $e_G\colon \Cc G \to \Cc$ is multiplicative on $\Cc G$, being the trivial character of the group $G$. Furthermore, since $\Cc$ carries the trivial action of the group $T$, the Giaquinto-Zhang twist of $e_G$ is the same linear map $e_G$, viewed as a representation $e_G\colon (\Cc G)_{\mathcal F}\to\Cc$
of the twisted algebra $(\Cc G)_{\mathcal F}$. Therefore, $e_G$ is multiplicative on $(\Cc G)_{\mathcal F}$. Thus $e_G \phi$ and $e_W$ are two algebra homomorphisms
from $\Cc W$ to $\Cc$; they both send grouplike generators $\sigma_i$, $i\in [n]$, and $t\in T(m,p,n)$ to $1$, hence coincide
on $\Cc W$.\bb

\nt We conclude that the above diagram commutes.
But it shows that $\phi$ intertwines the action $\underline \rhd$ 
of $\Cc W$ on $S_{-1}(V)$ with the 
$(\Cc G)_{\mathcal F}$-action arising from the 
representation of 
$$
H_{\mathcal F} = S(V)_{\mathcal F} \otimes (\Cc G)_{\mathcal F}
\otimes S(V^*)_{\mathcal F},
$$
on the space $S(V)_{\mathcal F}$, induced from the trivial $(\Cc G)_{\mathcal F}
S(V^*)_{\mathcal F}$-module. The latter is the Giaquinto-Zhang twist of the trivial $(\Cc G)S(V^*)$-module. Therefore, by Proposition~\ref{twisting_induced_rep_prop}, 
the $H_{\mathcal F}$-action on $S(V)_{\mathcal F}$ is the Giaquinto-Zhang twist of the standard $H$-module $S(V)$, arising from the trivial representation of $\Cc G$.
We know that the action of $\Cc G$ on this standard module
is the natural action of $\Cc G$ on $S(V)$, so the lemma is proved.
\end{proof}

\nt Our next step is to replace $S(V)$ and $S_{-1}(V)$ in Lemma~\ref{lem:twisted_action} by the respective coinvariant algebras. For this, we need to show the coinvariant algebras are
also defined on the same underlying vector space.

\begin{lemma} \label{lem:underlying_spaces}
The ideal $I_G$ of $S(V)$ and the ideal $I_W$ of $S_{-1}(V)$ share the same underlying subspace $I$ of $S$. This subspace is stable under the action of the group $T$.
\end{lemma}

\begin{proof}
We temporarily write $\cdot$ for $S(V)$-multiplication and  $\star$ for $S_{-1}(V)$-multiplication on the space $S$. Then $I_G$ is the $(\cdot)$-ideal generated by $p_1^{(m)},\dots,p_{n-1}^{(m)}$ and $r^{(\frac mp)}$ as given in \eqref{eq:invariant_generators}, and $I_W$ is the $(\star)$-ideal with the same generators. Hence $I_W$ is spanned by $M\star p_i^{(m)}\star N$ and $M\star r^{(\frac mp)} \star N$ where $M,N$ are standard monomials and $i\in [n]$. Using the facts that $p_i^{(m)}$ is $\star$-central as $m$ is even, that $r^{(\frac mp)}$ is a monomial, and that $M\star N=\pm M\cdot N$ for all monomials $M,N$, we conclude that all elements of this spanning set lie in $I_G$. So $I_W \subseteq I_G$. Similarly $I_G\subseteq I_W$. Note that $T=\langle \gamma_1,\dots,\gamma_n\rangle$ acts on the space $S$ by automorphisms, and also by algebra automorphisms on $S(V)$ and $S_{-1}(V)$. This action leaves $p_i^{(m)}$ invariant, and maps $r^{(\frac mp)}$ to $\pm r^{(\frac mp)}$, hence also maps $I_G$ to $I_G$.
\end{proof}

\begin{cor}\label{cor:twist} The algebras $\underline S_W$ and $S_G$ are defined on the same underlying vector space $S/I$.
As an algebra, $\underline S_W$ is the cocycle twist $(S_G)_{\mathcal F}$ of $S_G$. The isomorphism 
$\phi\colon \Cc W \to (\Cc G)_{\mathcal F}$ intertwines the 
action $\underline \rhd$ of $\Cc W$ on $\underline S_W$ and the Giaquinto-Zhang twist $\rhd_{\mathcal F}$ of the action 
$\rhd$ of $\Cc G$ on $S_G$.
\end{cor}
\begin{proof}
This follows from the fact that $S_{-1}(V)=S(V)_{\mathcal F}$ \eqref{eq:phi-identity} and
Lemma~\ref{lem:twisted_action}, by passing to the quotient
modulo $I=I_G=I_W$ (Lemma~\ref{lem:underlying_spaces}), which is 
an ideal in both algebras and is stable under all actions 
considered here.
\end{proof}

\subsection{The trace is unchanged under the twist.}

To prove Theorem~\ref{thm:nccoinv}, we calculate 
the character of the representation of $W$ afforded by $\underline S_W$, that is, the trace $\tr_{\underline S_W}(g\underline \rhd)$ of the action 
$\underline\rhd$ of $g\in W$ on $\underline S_W$.
By Corollary~\ref{cor:twist}, this is the same as 
$\tr_{(S_G)_{\mathcal F}}(\phi(g)\rhd_{\mathcal F})$. We will now show that neither the Giaquinto-Zhang twist by $\mathcal F$ nor the
application of $\phi$ changes the trace. The former follows from a 
general result about Giaquinto-Zhang twists:

\begin{proposition}\label{prop:twist_trace}
Let $H$ be a Hopf algebra with involutive antipode over a field $\Kk$, $A$ an $H$-module algebra, $V$ a finite-dimensional $(H,A)$-module where $A$ acts via $\rhd$, and $\mathcal F\in H\otimes H$ a $2$-cocycle. Denote the $\mathcal F$-twisted action of $A_{\mathcal F}$ on the space $V$ by $\rhd_{\F}$.\bb

\nt For any element $a$ of the vector space of $A$, $\tr_V(a\rhd) = \tr_V(a\rhd_{\mathcal F})$.
\end{proposition}
\begin{proof}
Denote the action of $h\in H$
on $a\in A$ simply by $h(a)$; same for the $H$-action on $V$ and 
on $V^*$.
Recall \cite[Proposition 9.3.3]{majid_1995} 
that the $H$-action on $V^*$ is such that  
the evaluation map $\langle\cdot,\cdot\rangle\colon V^*\otimes V \to \Kk$
and the coevaluation map $\Kk \to V\otimes V^*$, 
$1\mapsto \sum_i x_i \otimes y_i$, are $H$-module maps, where  
$\Kk$ is the trivial $H$-module.  Equivalently, for all $h\in H$, 
$x\in V$, $y\in V^*$,
\begin{equation}\label{eq:ev-coev}\textstyle
\langle y,h(x) \rangle = \langle S^{-1}h(y),x\rangle,\qquad  
\sum_i x_i \otimes h(y_i) = \sum_i S^{-1}h(x_i)\otimes y_i,
\end{equation}
where $S\colon H\to H$ is the antipode. Since $V$ is an $(H,A)$-module, it follows
from~\eqref{eq:compatibility} that for $a\in A$,
\begin{equation}\label{eq:h-act}
h(a)\rhd x = h_{(1)}( a\rhd Sh_{(2)}(x) ).
\end{equation}
By definition of trace, 
$\tr_V(a\rhd) = \sum_i \langle y_i, a\rhd x_i\rangle$. 
Writing $\mathcal F^{-1}\in H\otimes H$ as $f'\otimes f''$ (summation understood),
we compute the trace of the $\mathcal F$-twisted action of $a$:
\begin{align}
	\tr_V(a\rhd_{\mathcal F}) =  \sum_i \langle y_i, a\rhd_{\mathcal F} x_i\rangle \notag
	& = \sum_i \langle y_i, f'(a)\rhd f''(x_i)\rangle \notag
	\\ & = \sum_i \langle y_i, f'_{(1)}(a\rhd Sf'_{(2)}f''(x_i))\rangle \notag
	\\ & = \sum_i \langle S^{-1}f'_{(1)}(y_i), a\rhd Sf'_{(2)}f''(x_i)\rangle
	\label{eq:trace_calc}
\end{align} 
where we used definition \eqref{eq:def_twist} of $\rhd_{\mathcal F}$, 
\eqref{eq:ev-coev} and \eqref{eq:h-act}. 
We now manipulate the coevaluation tensor:
\begin{align*}
Sf'_{(2)}f''(x_i) \otimes S^{-1}f'_{(1)}(y_i) 
& = S^{-2}f'_{(1)}Sf'_{(2)}f''(x_i) \otimes y_i 
\\   & = f'_{(1)}Sf'_{(2)}f''(x_i) \otimes y_i
\\ & = \epsilon (f')f''(x_i) \otimes y_i
\end{align*}
where we applied the second part of~\eqref{eq:ev-coev}, the assumption that $S^2=\id_H$ and the antipode law. Here $\epsilon\colon H \to \Kk$ is the counit. 
Since by Definition~\ref{def:2-cocycle} $\mathcal F$ is counital, $(\epsilon \otimes \id_H)\mathcal F^{-1}=1$ 
which means $\epsilon (f')f''(x_i) \otimes y_i$ is equal to the coevaluation tensor 
$x_i \otimes y_i$. Thus, \eqref{eq:trace_calc} rewrites as 
$\sum_i \langle y_i, a\rhd x_i\rangle$ which is $\tr_V(a\rhd)$.
\end{proof}

\subsection{The trace is unchanged under \texorpdfstring{$\phi$}{}; proof of Theorem \ref{thm:nccoinv}.}

If $\Gamma$ is any finite group, consider the linear map 
$$
\chi_{\Cc \Gamma}\colon \Cc \Gamma \to \Cc,
\text{ defined on the basis $\{g\mid g\in\Gamma\}$ of $\Cc \Gamma$  by }
\chi_{\Cc \Gamma}(g)=\begin{cases} |\Gamma| & g=1_\Gamma,\\
		0 & g\neq 1_\Gamma.
\end{cases}
$$
Of course, we recognise $\chi_{\Cc \Gamma}$ as the character of the regular representation of $\Gamma$. 
Viewing $\phi\colon \Cc W\to (\Cc G)_{\mathcal F}$ as a linear map between the underlying vector spaces $\Cc W$, $\Cc G$, 
we show that $\phi$ intertwines the regular characters:
\begin{lemma}\label{lem:reg_char}
$\chi_{\Cc G} \circ \phi = \chi_{\Cc W}$.
\end{lemma}
\begin{proof}
We have $G=G(m,p,n)=\mathbb S_n \ltimes T(m,p,n)$, so elements of $G$ are written as $wt$ where $w\in \mathbb S_n$ and $t\in  T(m,p,n)$. By \cite[Proposition 5.7]{twistsrcas}, there exist elements $\theta_w\in \Cc T(2,1,n)$, such that $\phi^{-1}(wt)=w \theta_w t$ for all $w,t$. Moreover, the proof of \cite[Proposition 5.7]{twistsrcas} shows that $\theta_{1_{\mathbb S_n}}=1$.\bb

\nt So if $w\in \mathbb S_n$, $w\ne 1_{\mathbb S_n}$, and $t\in T(m,p,n)$, 
then $wt\ne 1_G$ so $\chi_{\Cc G}(wt)=0$. On the other hand, 
$\chi_{\Cc W}(\phi^{-1}(wt)) = \chi_{\Cc W}(w\theta_w t)=0$ because 
$w\ne 1_{\mathbb S_n}$ guarantees that $w\theta_w t$ is a linear combination of group elements of $W$, none of which is $1_W$.\bb

\nt In the case $w=1_{\mathbb S_n}$, $t\in T(m,p,n)$, we have 
$\phi^{-1}(t)=t$, hence $\chi_{\Cc G}(t)$ and $\chi_{\Cc W}(\phi^{-1}(t))$ are both $0$ if $t\ne 1_G$ and are both $|W|=|G|$ if $t=1_G$. This shows that $\chi_{\Cc W} \circ \phi^{-1}$ and $\chi_{\Cc G}$ agree on all elements of $G$, proving the Lemma.
\end{proof}

\nt We are ready to finish the proof of Theorem~\ref{thm:nccoinv}.

\begin{proof}[Proof of  Theorem~\ref{thm:nccoinv}]
We need to show that the trace, $\tr_{\underline S_W}(g\underline\rhd)$, of the action of $g\in W$ on $\underline S_W$, is equal to 
$\chi_{\Cc W}(g)$. By Corollary~\ref{cor:twist}, the trace equals
$\tr_{(S_G)_{\mathcal F}}(\phi(g)\rhd_{\mathcal F})$, which by Proposition~\ref{prop:twist_trace} coincides with 
$\tr_{S_G}(\phi(g)\rhd)$. By Chevalley's Theorem~\ref{thm:chevalley}, 
the latter is $\chi_{\Cc G}(\phi(g))$ which by Lemma~\ref{lem:reg_char}
is $\chi_{\Cc W}(g)$, as claimed.\bb

\nt Our application of Proposition~\ref{prop:twist_trace}
is valid because the Hopf algebra $\Cc T$ has involutive antipode.
Indeed, not only the antipode of any commutative or cocommutative Hopf algebra is involutive \cite[Corollary 7.1.11]{RadfordBook}, but the antipode of $\Cc T$ is, in fact, the identity map because 
every element of the group $T$ is self-inverse.
\end{proof}

\section{Twists of restricted rational Cherednik algebras}
\label{sec:last}
\subsection{Restricted rational Cherednik algebras.}
Let $\Kk$ be a field, $V$ a finite-dimensional vector space over $\Kk$, and 
$G\subset GL(V)$ a finite non-modular reflection group.
The classical coinvariant algebra of $G$ 
appears as a PBW-type factor in the algebra factorisation of the 
\define{restricted rational Cherednik algebra} $\overline H_c(G)$. 
These finite-dimensional quotients of rational Cherednik algebras $H_{0,c}(G)$ 
at $t=0$ originate from Gordon's 2003 paper \cite{gordon_2003} and have attracted 
considerable interest since.\bb

\nt Of interest to us is the case 
$G=G(m,p,n)$, and in this section we take $\Kk$ to be 
a subfield of $\Cc$ over which the reflection representation 
of $G(m,p,n)$ is defined; the smallest such $\Kk$ is therefore 
the $m$th cyclotomic extension of $\Q$.
The $\Kk$-rational Cherednik algebra $H_{0,c}(G)$, given by Definition~\ref{rational_cherednik_defn}, is defined over $\Kk$ as long as the parameter $c$ is $\Kk$-valued.
Thiel \cite{thiel} defines the restricted rational Cherednik algebra as 
$$\overline H_c(G) = H_{0,c}(G)/ (I,I')$$
where $I$ is the ideal generated by homogeneous $G$-invariants of positive degree in $S(V)$, $I'$ is the similar ideal in $S(V^*)$, and $(I,I')$ is the ideal generated by $I$ and $I'$ in $H_{0,c}(G)$. One has the following algebra factorisation,
$$
\overline H_c(G) \cong S(V)/I \otimes \Kk G \otimes S(V^*)/I' = S_G \otimes \Kk G \otimes S_G',
$$
where $S_G$ is the coinvariant algebra of $G$ from the previous section, and $S_G'$ is defined in the same way from the action of $G$ on the dual space $V^*$.

\subsection{Twists of restricted rational Cherednik algebras.}
Let now $W=\mu(G(m,p,n))$ be the mystic reflection group 
which is the mystic counterpart of $G$.
We define a finite-dimensional quotient of the negative braided Cherednik algebra 
$\underline{H}_{0,\underline{c}}(W)$ by
$$
\overline{\underline{ H}}_{\underline c}(W) 
= \underline{H}_{0,\underline{c}}(W) / (I_W,I_W')
$$
where $I_W\subseteq S_{-1}(V)$, $I_W'\subseteq S_{-1}(V^*)$
are ideals generated by homogeneous invariants of positive degree. This can be called the ``restricted negative braided Cherednik algebra". The results of Lemma \ref{lem:underlying_spaces}, Corollary \ref{cor:twist} and \cite[Theorem 5.2]{twistsrcas} immediately lead us to the following, 
\begin{theorem}
The cocycle twist of $\overline H_c(G)$ by the cocycle $\F$ given in \eqref{cocycle_f} is isomorphic to $\underline{\overline{H}}_{\underline{c}}(W)$.
This algebra has triangular decomposition
$$
\underline{\overline{H}}_{\underline{c}}(W)\cong S_{-1}(V)/I_W \otimes \Kk G \otimes S_{-1}(V^*)/I_W' = \underline S_W \otimes \Kk G \otimes \underline S_W'
$$
into the group algebra of $W$ and two noncommutative coinvariant algebras.
\end{theorem}
\nt We further observe the following, 
\begin{proposition}
If $\frac mp$ is even, then $\underline{\overline{H}}_{\underline{c}}(W)$
and $\overline H_c(G)$ are isomorphic as algebras. 
\begin{proof}
Indeed, Proposition~\ref{twist_isom_prop} applies because 
the action of the Hopf algebra $\Kk T$ 
on $\overline{H}_c(G(m,p,n))$ is adjoint, as it factors via
the embedding of $T$ as the subgroup $T(2,1,n)$ in $G(m,p,n)$. 
\end{proof}
\end{proposition}

\subsection{A case where the restricted rational Cherednik algebra is not isomorphic to its twist.}

We stress that Proposition~\ref{twist_isom_prop} establishes an isomorphism between $A_{\mathcal F}$ and $A$ 
over an arbitrary field $\Kk$, as long as the $H$-module structure on $A$ and the cocycle $\mathcal F$ are defined over $\Kk$, and the $H$-action on $A$ is adjoint.
We will now give an explicit example where the action is not adjoint, and the restricted rational Cherednik algebra
is not isomorphic to its twist by~$\mathcal F$.\bb 

\nt Of course, we need $\frac mp$ to be odd, and we consider the case $m=p=n=2$. Note $G(2,2,2)$ is the Klein four-group generated by the commuting involutions $s_{12}$ and $\bar s_{12}$. On the other hand, $\mu(G(2,2,2))$ is the cyclic group of order $4$ generated by the mystic reflection $\sigma_{12}=\begin{pmatrix}0&-1\\ 1&0\end{pmatrix}$.  We see $\Kk$ the minimal field over which the reflection representation of $G(2,2,2)$ and the mystic reflection representation of $\mu(G(2,2,2))$ are defined is the rationals, $\Kk=\Q$.\bb

\nt The proof of the following result is partly based 
on an explicit computer calculation (over the rational 
field).

\begin{proposition}\label{prop:not_isom}
If $c=1$,
then the $\Q$-algebras $\overline H_c(G(2,2,2))$ and 
$\overline{\underline{H}}_{-c}(\mu(G(2,2,2)))$
are not isomorphic.
\begin{proof}
Note that $G(2,2,2)$ is the non-reduced Coxeter group of type 
$A_1\times A_1$, so 
$$
\overline H_c(G(2,2,2)) \cong \overline  H_c(G(2,1,1)) \otimes \overline H_c(G(2,1,1))
$$
where $G(2,1,1)$ is the reflection group $\{\pm 1\}$ 
of order $2$ which operates on the $1$-dimensional space.
Thus, $\overline H_c(G(2,2,2))$ is a $64$-dimensional 
algebra which is a tensor product of two (commuting) 
$8$-dimensional subalgebras.
Explicitly, $\overline  H_c(G(2,1,1))$ is presented as
$$
\langle y,s,x\mid s^2=1,\ sx=-xs,\ ys=-sy,\ 
yx = xy-2cs, \ x^2=y^2=0\rangle. 
$$
It follows that 
$$
Z(\overline H_c(G(2,2,2)))  \cong Z( \overline  H_c(G(2,1,1))) \otimes Z( \overline H_c(G(2,1,1)) ).
$$
By results of \cite{hird2023}, the centre $Z( \overline H_c(G(2,1,1)) )$ is two-dimensional. Consider $z=xy-cs$.
We claim that $z$ is central in $\overline H_c(G(2,1,1)) )$. Indeed, $sz=zs$, and 
$$
zx = xyx-csx = x(xy-2cs) +cxs = 0-cxs = -cxs, \quad 
xz =-cxs = -cxs
$$
so $z$ commutes with $x$. Similarly, $z$ commutes with $y$.
We conclude that $Z( \overline H_c(G(2,1,1)) )=\Q 1 + \Q z$. Observe that
$$
z^2 = zxy - czs = (-cxs)y +cxsy + c^2 =c^2.
$$
Since $c=1$ in our example, $z^2=1$. In particular, 
the algebra $Z( \overline H_c(G(2,1,1)) )$
does not contain elements of multiplicative order~$4$, 
hence such elements do not exist in 
$Z(\overline H_c(G(2,2,2)))$.
%
%
On the other hand, we claim that
\begin{align*}
\gamma = & \frac{1}{4}\big[ 2c^2 -2c\sigma + 2c^2\sigma^2 + 2c\sigma^3 \\ 
 & + cx_1\sigma y_1 -c x_1 \sigma y_2 - 2 x_1 \sigma^2 y_2 + c x_2 \sigma y_1\\ 
& + c x_2 \sigma y_2 + 2 x_2 \sigma^2 y_1 + 2 x_1^2 y_1^2 + c x_1 \sigma^3 y_1\\
 & + c x_1 \sigma^3 y_2 - c x_2 \sigma^3 y_1 + c x_2 \sigma^3 y_2\big]
\end{align*}
is a central element of $\overline{\underline{H}}_{c}(\mu(G(2,2,2)))$. Moreover, $\gamma^2$ is not a scalar whereas $\gamma^4 = c^4$. This is a result of a computer calculation. Hence, if $c=\pm 1$, the centre of 
$\overline{\underline{H}}_{c}(\mu(G(2,2,2)))$
contains an element of multiplicative order $4$, which shows that the two algebras are not isomorphic over $\Q$.
\end{proof}
\end{proposition}

\nt We remark that the above example does not work over $\Cc$, or indeed over an extension of $\Q$ containing $\sqrt{-1}$.
If scalars are extended to $\Q(\sqrt{-1})$, the two algebras 
become isomorphic. We currently have no explicit example 
where we could disprove isomorphism over $\Cc$.\bb

\nt We finish with the following conjecture.

\begin{conj}
Let $m$ be even and $\frac mp$ be odd, 
and let $k=\Q(\sqrt[m]{1})$. The $\Kk$-algebras $\underline{\overline{H}}_{\underline{c}}(\mu(G(m,p,n)))$ and $\overline H_c(G(m,p,n))$ are not isomorphic. 
\end{conj}

\bibliographystyle{plain}
\bibliography{mybib}

\begin{thebibliography}{10}

\bibitem{mystic_reflections}
Y.~Bazlov and A.~Berenstein.
\newblock Mystic reflection groups.
\newblock {\em Symmetry, Integrability and Geometry: Methods and Applications},
  10, April 2014.

\bibitem{twistsrcas}
Y.~Bazlov, A.~Berenstein, E.~Jones-Healey, and A.~Mcgaw.
\newblock {Twists of rational Cherednik algebras}.
\newblock {\em The Quarterly Journal of Mathematics}, 74(2):511--528, 2022.

\bibitem{berenstein_schmidt}
A.~Berenstein and K.~Schmidt.
\newblock Factorizable module algebras.
\newblock {\em International Mathematics Research Notices},
  2019(21):6711--6764, 2019.

\bibitem{BGG_1971}
I.~N. Bernstein, I.~M. Gelfand, and S.~I. Gelfand.
\newblock Structure of representations that are generated by vectors of highest
  weight.
\newblock {\em Funckcional. Anal. i Prilo\v{z}en.}, 5(1):1--9, 1971.

\bibitem{Broer2011}
A.~Broer, V.~Reiner, L.~Smith, and P.~Webb.
\newblock Extending the coinvariant theorems of {C}hevalley, {S}hephard-{T}odd,
  {M}itchell, and {S}pringer.
\newblock {\em Proceedings of the London Mathematical Society},
  103(5):747--785, 2011.

\bibitem{Chevalley1955}
C.~Chevalley.
\newblock Invariants of finite groups generated by reflections.
\newblock {\em American Journal of Mathematics}, 77(4):778--782, 1955.

\bibitem{etingof2015tensor}
P.~Etingof, S.~Gelaki, D.~Nikshych, and V.~Ostrik.
\newblock {\em Tensor Categories}.
\newblock Mathematical Surveys and Monographs. American Mathematical Society,
  2016.

\bibitem{alma9920315464401631}
M.~Geck and G.~Pfeiffer.
\newblock {\em Characters of finite Coxeter groups and Iwahori-Hecke algebras}.
\newblock London Mathematical Society monographs ; 21. Clarendon Press, Oxford,
  2000.

\bibitem{giaquinto1998bialgebra}
A.~Giaquinto and J.~J. Zhang.
\newblock Bialgebra actions, twists, and universal deformation formulas.
\newblock {\em Journal of Pure and Applied Algebra}, 128(2):133--151, 1998.

\bibitem{gordon_2003}
I~Gordon.
\newblock Baby {V}erma modules for rational cherednik algebras.
\newblock {\em Bulletin of the London Mathematical Society}, 35(3):321--336,
  2003.

\bibitem{higman_1955}
D.~G. Higman.
\newblock Induced and produced modules.
\newblock {\em Canadian J. Math.}, 7:490--508, 1955.

\bibitem{hird2023}
N.~Hird.
\newblock An explicit presentation of the centre of the restricted rational
  cherednik algebra.
\newblock {\em arXiv preprint arXiv:2301.03316}, 2023.

\bibitem{alma992976864702801631}
I.~M. Isaacs.
\newblock {\em Character theory of finite groups}.
\newblock Pure and applied mathematics, a series of monographs and textbooks ;
  v. 69. Academic Press, New York, 1976.

\bibitem{Kirkman2008ShephardToddChevalleyTF}
E.~Kirkman, J.~Kuzmanovich, and J.~J. Zhang.
\newblock {S}hephard-{T}odd-{C}hevalley theorem for skew polynomial rings.
\newblock {\em Algebras and Representation Theory}, 13:127--158, 2008.

\bibitem{kulish_2011}
P.~Kulish and A.~Mudrov.
\newblock Twisting adjoint module algebras.
\newblock {\em Lett. Math. Phys.}, 95(3):233--247, 2011.

\bibitem{alma9930780234401631}
G.~Lehrer and D.~E. Taylor.
\newblock {\em Unitary reflection groups}.
\newblock Australian Mathematical Society lecture series ; 20. Cambridge
  University Press, Cambridge, 2009.

\bibitem{majid_1995}
S.~Majid.
\newblock {\em Foundations of Quantum Group Theory}.
\newblock Cambridge University Press, 1995.

\bibitem{majid_2002}
S.~Majid.
\newblock {\em A Quantum Groups Primer}, volume 292 of {\em London Mathematical
  Society Lecture Note Series}.
\newblock Cambridge University Press, Cambridge, 2002.

\bibitem{montgomery_1993}
S.~Montgomery.
\newblock {\em Hopf algebras and their actions on rings}, volume~82 of {\em
  CBMS Regional Conference Series in Mathematics}.
\newblock Published for the Conference Board of the Mathematical Sciences,
  Washington, DC; by the American Mathematical Society, Providence, RI, 1993.

\bibitem{NeuselSmith2002}
M.~D. Neusel and L.~Smith.
\newblock {\em Invariant theory of finite groups}, volume~94 of {\em
  Mathematical Surveys and Monographs}.
\newblock American Mathematical Society, Providence, RI, 2002.

\bibitem{RadfordBook}
D.~E. Radford.
\newblock {\em Hopf Algebras}.
\newblock World Scientific, 2011.

\bibitem{rieffel_1975}
M.~A. Rieffel.
\newblock Induced representations of rings.
\newblock {\em Canadian Journal of Mathematics}, 27(2):261–270, 1975.

\bibitem{thiel}
U.~Thiel.
\newblock Restricted rational {C}herednik algebras.
\newblock In {\em Representation theory---current trends and perspectives}, EMS
  Ser. Congr. Rep., pages 681--745. Eur. Math. Soc., Z\"{u}rich, 2017.

\bibitem{verma_1966}
D-N. Verma.
\newblock {Structure of certain induced representations of complex semisimple
  Lie algebras}.
\newblock {\em Bulletin of the American Mathematical Society}, 74(1):160 --
  166, 1968.

\end{thebibliography}

\end{document}